\documentclass[bj, preprint]{imsart}
\RequirePackage{natbib}

\usepackage[utf8]{inputenc}

\RequirePackage[colorlinks,citecolor=blue,urlcolor=blue]{hyperref}

\usepackage{amsmath}
\usepackage{amsfonts}
\usepackage{amssymb}
\usepackage{amsthm}

\usepackage{dsfont}

\usepackage[norelsize,ruled,vlined,commentsnumbered]{algorithm2e}

\usepackage{graphicx}
\usepackage{caption}
\usepackage{subcaption}

\startlocaldefs
\newtheorem{theorem}{Theorem}
\newtheorem{corollary}[theorem]{Corollary}
\newtheorem{prop}[theorem]{Proposition}
\newtheorem{lemma}[theorem]{Lemma}
\newtheorem*{remark}{Comment}
\newtheorem*{condition}{Condition}

\newcommand{\Acal}{\mathcal{A}}
\newcommand{\Bbf}{\mathbf{B}}
\newcommand{\Bcal}{\mathcal{B}}
\newcommand{\Cbf}{\mathbf{C}}
\newcommand{\Ccal}{\mathcal{C}}
\newcommand{\Cfrak}{\mathfrak{C}}

\newcommand{\Dfrak}{\mathfrak{D}}
\newcommand{\Ebb}{\mathbb{E}}
\newcommand{\fbf}{\mathbf{f}}
\newcommand{\Fcal}{\mathcal{F}}

\newcommand{\hbf}{\mathbf{h}}
\newcommand{\ibf}{\mathbf{i}}
\newcommand{\jbf}{\mathbf{j}}
\newcommand{\kbf}{\mathbf{k}}

\newcommand{\Lbf}{\mathbf{L}}

\newcommand{\Mbf}{\mathbf{M}}
\newcommand{\mbf}{\mathbf{m}}
\newcommand{\Mcal}{\mathcal{M}}
\newcommand{\Nbb}{\mathbb{N}}
\newcommand{\Nbf}{\mathbf{N}}
\newcommand{\Nfrak}{\mathfrak{N}}
\newcommand{\Obf}{\mathbf{O}}
\newcommand{\Pbb}{\mathbb{P}}
\newcommand{\Pbf}{\mathbf{P}}

\newcommand{\Pfrak}{\mathfrak{P}}
\newcommand{\Qbf}{\mathbf{Q}}
\newcommand{\Qcal}{\mathcal{Q}}
\newcommand{\Rbb}{\mathbb{R}}
\newcommand{\Rbf}{\mathbf{R}}
\newcommand{\Sbb}{\mathbb{S}}
\newcommand{\Sbf}{\mathbf{S}}
\newcommand{\Sfrak}{\mathfrak{S}}
\newcommand{\Tcal}{\mathcal{T}}
\newcommand{\Ubf}{\mathbf{U}}
\newcommand{\Ucal}{\mathcal{U}}

\newcommand{\Xcal}{\mathcal{X}}
\newcommand{\Ycal}{\mathcal{Y}}

\newcommand{\argmin}{\text{arg min}} 
\newcommand{\cinf}{C_{\Fcal,\infty}}
\newcommand{\cquad}{C_{\Fcal,2}}

\newcommand{\dperm}{d_\text{perm}}
\newcommand{\Id}{{\text{Id}}}
\newcommand{\Ker}{{\text{Ker}}}
\newcommand{\one}{\mathds{1}}
\newcommand{\pen}{\text{pen}}
\newcommand{\pentilde}{\widetilde{\pen}}
\newcommand{\Span}{\text{Span}}
\endlocaldefs

\setattribute{journal}{name}{}

\begin{document}

\begin{frontmatter}

\title{Consistent order estimation for nonparametric Hidden Markov Models}
\runtitle{Nonparametric HMM order estimation}


\author{\fnms{Luc} \snm{Leh\'ericy}\ead[label=e1]{luc.lehericy@math.u-psud.fr}}
\address{Laboratoire de Mathématiques d'Orsay, Univ. Paris-Sud, CNRS, Université Paris-Saclay, 91405 Orsay, France.
\printead{e1}}

\runauthor{L. Leh\'ericy}

\begin{abstract}
We consider the problem of estimating the number of hidden states (the \emph{order}) of a nonparametric hidden Markov model (HMM). We propose two different methods and prove their almost sure consistency without any prior assumption, be it on the order or on the emission distributions. This is the first time a consistency result is proved in such a general setting without using restrictive assumptions such as \textit{a priori} upper bounds on the order or parametric restrictions on the emission distributions. Our main method relies on the minimization of a penalized least squares criterion. In addition to the consistency of the order estimation, we also prove that this method yields rate minimax adaptive estimators of the parameters of the HMM - up to a logarithmic factor.
Our second method relies on estimating the rank of a matrix obtained from the distribution of two consecutive observations.
Finally, numerical experiments are used to compare both methods and study their ability to select the right order in several situations.
\end{abstract}



\end{frontmatter}

\section{Introduction}

\subsection{Context and motivation}

Hidden Markov models (HMM in short) are powerful tools to study time-evolving processes on heterogeneous populations. Nonparametric HMMs\---that is, hidden Markov models where the parameters are not restricted to a finite-dimensional space\---have proved useful in a wide range of applications, see for instance \cite{CC00} for voice activity detection, \cite{LWM03} for climate state identification, \cite{Lef03} for automatic speech recognition, \cite{SC09} for facial expression recognition, \cite{VBMMR14} for methylation comparison of proteins, \cite{YPRH11} for copy number variants identification in DNA analysis.

In practice, the hidden states often have an interpretation in the modelling of the phenomenon. It is thus important to be able to infer the right order in addition to the parameters when dealing with hidden Markov models. However, this task is notoriously difficult: \cite{gassiat2000likelihood} show that the likelihood ratio statistic is unbounded even in the simple case where one wants to test if a HMM has 1 or 2 hidden states. As far as we know, no consistency result has been proved about order selection for nonparametric HMMs. Even for parametric HMMs, no estimator has been proved to be consistent in a general setting without assuming that an \emph{a priori} upper bound on the order is known beforehand.

Not only is the order estimation useful in order to interpret the model, it is also necessary to ensure stability. This is because over estimating the order causes a loss of identifiability: there are several ways to add one state to a HMM without changing anything to its distribution. The spectral estimators (\cite{AHK12, dCGLLC15}) are especially sensitive to this problem, as shown by~\cite{Leh15} and Figure~\ref{fig_estimators_N19998}: as soon as the HMM becomes close to a HMM with fewer hidden states, the estimators give absurd results. Thus, estimating the right order is crucial for such methods to be effective.

Formally, a hidden Markov model is a markovian process $(X_t, Y_t)_{t \geq 1}$ taking value in $\Xcal \times \Ycal$. $(X_t)_{t \geq 1}$ is a Markov chain and the observations $Y_t$ depend only on the associated $X_t$ (i.e. the $(Y_t)_{t \geq 1}$ are independent conditionally on $(X_t)_{t \geq 1}$). The states $(X_t)_{t \geq 1}$ are assumed to be hidden, so that one has only access to the observations $(Y_t)_{t \geq 1}$. When the number of hidden states $|\Xcal|$ (which we call the \emph{order} of the HMM) is finite, the model is completely defined by its order, the initial distribution and the transition matrix of the hidden Markov chain, and the possible distributions of an observation $Y_t$ conditionally to the values of its hidden state $X_t$, which we call the \emph{emission distributions}.
The goal of the estimation procedures is to recover these parameters by using only the observations $(Y_t)_{t \geq 1}$.

Up to now, most theoretical results on hidden Markov models dealt with the parametric frame, that is with a finite number of parameters. However, it is not always possible to restrict the model to such a convenient finite-dimensional space. Theoretical results in the nonparametric framework were only developed recently and do not address the order estimation problem. \cite{dCGL15} propose an adaptive quasi-rate minimax least squares method. \cite{dCGLLC15} and \cite{robin2014estimating} study spectral methods. The latter is also proved to reach the minimax convergence rate but is not adaptive: it requires the regularity of the emission distributions to be known. All these methods require the order of the HMM to be known.

Our work is novel on three points. First, it deals with the nonparametric setting: we need no parametric or regularity assumption on the emission densities. Note that all our results also apply to parametric settings or even to finite observation spaces, since these are just special cases of nonparametric estimation. Secondly, we do not require any \emph{a priori} upper bound on the order, an assumption that is often made in earlier works, both frequentist and bayesian. Finally, our least squares method yields estimators of all model parameters at the same time, without requiring any prior information. Oracle inequalities show that these estimators are rate minimax adaptive up to a logarithmic factor.

\subsection{Related works}

The first step to obtain theoretical results was to understand when hidden Markov models are identifiable. This challenging issue was only solved a few years ago, see \cite{GCR15} (following \cite{allman2009identifiability} and \cite{HKZ12}) and with weaker assumptions \cite{AH14}. Both proved that under generic assumptions, the parameters of the HMM can be recovered from the distribution of a finite number of consecutive observations, thus paving the way for guarantees on parameter estimation.

HMM inference is generally decomposed in two parts. The first one is the estimation of the order, and the second one is the estimation of the parameters once the order is known.


From a theoretical point of view, the order estimation problem remains widely open in the HMM framework. One can distinguish two kinds of results. The first kind does not need an \emph{a priori} upper bound on the order, but is only applicable to restrictive cases. For instance, using tools from coding theory, \cite{gassiat2003optimal} introduced a penalized maximum likelihood order estimator for which they prove strong consistency without \emph{a priori} upper bound on the order of the HMM. Nevertheless, their result is restricted to a finite observation space and they have to use heavy penalties that grow as a power of the order. For the special case of Gaussian or Poisson emission distributions, \cite{chambaz2009minimum} showed that the penalized maximum likelihood estimator is strongly consistent without any \emph{a priori} upper bound on the order. The second kind of results is more general but requires an \emph{a priori} upper bound of the order just to get weak consistency of order estimators, for penalized likelihood criterion (\cite{gassiat2002likelihood}) as well as Bayesian approaches (\cite{gassiat2014posterior, van2016overfitting}).

On a practical side, several order estimation methods using penalized likelihood criterion have been studied numerically, see for instance \cite{VBMMR14} when emission distributions are a mixture of parametric densities or \cite{celeux2008selecting} for parametric HMMs. The latter also introduced cross-validation procedures that aimed for circumventing the lack of independance of the observations. In the case of nonparametric HMMs, \cite{langrock2015nonparametric} studied a method using P-splines with a custom penalization.


Then comes the question of estimating the parameters of the HMM once its order is known. In the parametric setting, the asymptotic behaviour of the maximum likelihood estimator is rather well understood (see for instance \cite{bickel1998asymptotic} or \cite{douc2004asymptotic} using techniques from \cite{legland-mcss00a}), but so far the question of its nonasymptotic behaviour remains open.
\cite{HKZ12} and \cite{AHK12} proposed a spectral method for parametric HMMs based on joint diagonalization of a set of matrices and controlled its nonasymptotic error. \cite{robin2014estimating} and \cite{dCGLLC15} extended this method to the nonparametric setting, and \cite{dCGL15} used the latter to obtain an estimator of the transition matrix of the hidden chain for a quasi-rate minimax adaptive least squares estimator of the emission densities.
Our least squares estimation method is a generalization of their procedure that is able to deal with all parameters at once and does not require auxiliary estimators.

\subsection{Contribution}


The aim of our paper is twofold. Firstly, we introduce two estimators of the order for nonparametric HMMs and show that both converge almost surely to the right order under minimal assumptions. Secondly, we numerically assess their ability to select the right order and compare their efficiency.\\

Our first and main method is the penalized least squares estimator.
This method is based on estimating the projection of the emission distributions onto a family of nested parametric subspaces. Our results hold for any Hilbert space, including parametric sets of emission densities and finite observation spaces. Then, for each subspace and for each possible value $K$ of the order, we look for the HMM with $K$ hidden states and with emission distributions in the chosen subspace that matches the observations ``best''\---where ``best'' means minimizing the empirical equivalent of an $\Lbf^2$ distance. This step provides an empirical distance between the observations and the model, which is then penalized in order to counterbalance the overfitting phenomenon that occurs when considering large models. Our first main result is that for a suitable choice of the penalty, choosing the model (i.e. the order and the subspace) which minimizes this penalized distance leads to a strongly consistent estimator of the order, see Corollary \ref{cor_consistency}.

In addition, this method also provides estimators of the other parameters of the HMM for free, by taking the parameters of the HMM corresponding to the selected model. We prove an oracle inequality on the $\Lbf^2$ risk of these estimators, which shows that they achieve the minimax adaptive rate of convergence, up to a logarithmic term, see Theorem \ref{th_oracle_param} and Corollary \ref{cor_oracle_even_when_mistaken}.


Our second estimator comes from spectral methods. Just like for our least squares procedure, we consider a nested family of parametric subspaces of a Hilbert space. Let us choose one of them, and denote by $(\varphi_a)_a$ an orthonormal basis of this subspace. Then, consider the matrix $\Nbf$ defined by
\begin{equation*}
\Nbf(a,b) := \Ebb[ \varphi_a(Y_1) \varphi_b(Y_2) ].
\end{equation*}
This matrix contains the coordinates of the density of $(Y_1, Y_2)$ in the orthonormal basis $(\varphi_a \otimes \varphi_b)_{a,b}$. It is proved in Section \ref{sec_spectral_theory} that the rank of $\Nbf$ is exactly equal to the order of the HMM as soon as the subspace is large enough. Therefore, finding its rank means finding the number of hidden states. However, in practice, one only has access to an empirical version of this matrix. The difficulty comes from the fact that this noisy version will almost surely have full rank. Thus, the key point is to recover the order of the true matrix given its empirical (full rank) counterpart. We achieve this by thresholding the spectrum of the empirical matrix.
Notice that other methods exist to estimate the rank of a matrix based on a noisy observation, see for instance \cite{kleibergen2006generalized} and references therein. Unfortunately, most can not be applied directly to our setting since they require an invertibility condition on the covariance matrix of the matrix entries. The CRT statistics from \cite{robin2000tests} is a notable exception, however their test of rank also requires the tuning of an unknown parameter in order to be weakly consistent.
\\

Then, we run an implementation of these two methods and compare their efficiency on simulated data. The difficulty at this stage comes from the fact that both method involve an unknown tuning parameter. This is a common issue that appears in every model selection method in one form or another, and many heuristics have been proposed to circumvent this difficulty.

For the least squares estimator, we compare two methods which have been both proved to be theoretically valid in simple cases and empirically validated in a large variety of situations: the slope heuristics (see for instance \cite{BMM12a} and references therein) and the dimension jump heuristics (introduced and proved to lead to an optimal penalization in the gaussian model selection framework by \cite{birge2007minimal}). Both behave well with our estimator and lead to a satisfying calibration of the penalty.

For the spectral estimator, we introduce a custom heuristics based on the fact that the smallest singular values of the empirical version of the matrix $\Nbf$ decrease in a simple manner. It is thus possible to calibrate an entirely data-driven threshold to distinguish ``significant'' singular values\---that is, the ones corresponding to non-zero singular values of the real $\Nbf$\---from noise.

The numerical validation shows that our least squares method performs well in almost any situation. It is able to select the right order accurately with notably fewer observations than the spectral estimator, and is easier to calibrate. On the other hand, the spectral method is very fast, which allows to take more observations into account. This allows to obtain satisfying estimators in a short amount of time.

Regarding the inference of the other parameters, our least squares estimator offers several advantages when compared to previous methods. First, it does not need a preliminary estimation of the transition matrix or of the order, unlike \cite{dCGL15} who used the transition matrix given by spectral estimators. Nevertheless, our method still reaches the adaptive minimax convergence rate for the estimation of the emission densities, up to a logarithmic factor. This is especially useful to avoid the cases where their auxiliary estimator fails. For instance, the spectral method that \cite{dCGL15} used is unreliable when the order is over estimated or where the states are almost linearly dependent, see for instance \cite{Leh15} or Figure \ref{fig_estimators_N19998}. Then, our least squares method is robust to an overestimation of the order, both theoretically and numerically, thanks to the iterative initialization procedure that we introduce. This initialization method consists in using estimators from smaller models as initial point for the minimization algorithm in order to avoid getting stuck in suboptimal local extrema. We believe it can be of practical interest since it produces robust estimators and can also be used in other settings, for instance as initialization for expectation maximization algorithm for maximum likelihood estimators.

\subsection{Outline of the paper}

Our paper is organized as follows.

Section \ref{sec_definitions} is devoted to the notations, the model and the assumptions.

Our main procedure, the penalized least squares method, is introduced in Section \ref{sec_least_squares}. We first state an identifiability proposition which we use to prove strong consistency of the estimator of the order. This is done in two steps. Firstly, we control the probability to underestimate the order. This is done thanks to Proposition \ref{prop_identifiability}, and gives an exponential bound on the probability of error, see Theorem \ref{th_underestimation}. Secondly, we control the probability to overestimate the order, see Theorem \ref{th_overestimation}. For this, we introduce a general condition on the penalty, which we use to prove polynomial decrease rate, and illustrate how to easily satisfy this condition. Finally, we state oracle inequalities on the estimators of the density of $L$ consecutive observations and on the parameters of the hidden Markov model under a generic assumption, see Theorem \ref{th_oracle_param} and Corollary \ref{cor_oracle_even_when_mistaken}, which shows that they reach the minimax convergence rate up to a logarithmic factor.

In Section \ref{sec_spectral_theory}, we introduce the spectral algorithm and propose a strongly consistent estimator of the order. This is done by thresholding the spectrum of the empirical version of the matrix $\Nbf$, which describes the projection of the distribution of two observations on an orthonormal basis, see Theorem \ref{th_spectr_consistent}.

In Section \ref{sec_numerical_experiments}, we propose practical algorithms to apply both methods and compare them.
Firstly, we set the parameters on which we will test both procedures.
Secondly, we compare their results and discuss their performance.
Lastly, we introduce and discuss the heuristics we used to practically implement both methods.

Our main technical result, Lemma \ref{lemma_control_of_Z}, can be found at the beginning of Section \ref{sec_proofs}. It is used extensively for both the consistency of the estimator of the order and the oracle inequalities on the HMM parameters. The rest of this section is dedicated to the proofs of the results.

Appendix \ref{app_algo} of our supplementary material contains the spectral algorithm from \cite{dCGLLC15} and \cite{dCGL15} that we use in our simulations. Appendix \ref{app_oracle_proof} gathers the proofs of Section \ref{sec_oracles}, which deals with the oracle inequalities for the least squares method. Finally, Appendix \ref{sec_proof_control_Z} contains the proof of Lemma \ref{lemma_control_of_Z}, and Appendix \ref{app_auxiliary} contains miscellaneous lemmas and proofs.


\section{Definitions and assumptions}
\label{sec_definitions}

We will use the following notations throughout the paper.

\begin{itemize}
\item $\Nbb^* = \{1, 2, \dots \}$ is the set of positive integers.
\item For $k \in \Nbb^*$, $[k]$ is the set $\{1, \dots, k\}$.
\item If $f_1$ and $f_2$ are two functions, we denote by $f_1 \otimes f_2$ their tensor product, defined by $f_1 \otimes f_2 (x_1, x_2) = f_1(x_1) f_2(x_2)$.
\item $\Span(a)$ is the linear space spanned by the family $a$.
\item If $E_1$ and $E_2$ are two linear spaces, we denote by $E_1 \otimes E_2$ their tensor product, that is the linear space spanned by the tensor products of their elements: $E_1 \otimes E_2 = \Span(f_1 \otimes f_2 | f_1 \in E_1, f_2 \in E_2)$.
\item $\Delta_K = \{ \pi \in [0,1]^K | \sum_{k=1}^K \pi_k = 1 \}$ is the simplex in dimension $K$. It will be seen as the set of probability measures on a finite set of size $K$.
\item $\Qcal_K \subset \Rbb^{K \times K}$ is the set of irreducible transition matrices of size $K$.
\item $\Id_K$ is the identity matrix of size $K$.
\item $\Lbf^2(A, \nu)$ is the Hilbert space of square integrable functions on $A$ with respect to the measure $\nu$.
\item The notation $C \equiv C(a,b,\dots)$ for a constant $C$ will mean that the value of $C$ depends on the specified parameters $a$, $b$, $\dots$ For several constants depending on the same parameters, we will write $(C,D) \equiv (C,D)(a,b,\dots)$.
\end{itemize}

In the following, $L$ is a positive integer which will denote the number of consecutive observations used for the estimation procedure.

\subsection{Hidden Markov models}

Let $(X_j)_{j \geq 1}$ be a Markov chain with finite state space $\Xcal$ of size $K^*$ with transition matrix $\Qbf^*$ and initial distribution $\pi^*$. Without loss of generality, we can set $\Xcal = [K^*]$.

Let $(Y_j)_{j \geq 1}$ be random variables on a measured space $(\Ycal, \mu)$ with $\mu$ $\sigma$-finite such that conditionally on $(X_j)_{j \geq 1}$ the $Y_j$'s are independent with a distribution depending only on $X_j$. Let $\nu^*_k$ be the distribution of $Y_j$ conditionally to $\{X_j = k\}$. Assume that $\nu^*_k$ has density $f^*_k \in \Lbf^2(\Ycal, \mu)$ with respect to $\mu$. We call $(\nu^*_k)_{k \in \Xcal}$ the \emph{emission distributions} and $\fbf^* = (f^*_1, \dots, f^*_{K^*})$ the \emph{emission densities}.

Then $(X_j, Y_j)_{j \geq 1}$ is a hidden Markov model with parameters $(\pi^*, \Qbf^*, \fbf^*, K^*)$. The hidden chain $(X_j)_{j \geq 1}$ is assumed to be unknown, so that the estimator only has access to the observations $(Y_j)_{j \geq 1}$. 

For $K \in \Nbb^*$, $\pi \in \Rbb^K$, $\Qbf \in \Rbb^{K \times K}$ and $\fbf \in (\Lbf^2(\Ycal, \mu))^K$, let
\begin{equation*}
g^{\pi, \Qbf, \fbf, K} = \sum_{k_1, \dots, k_L = 1}^K \pi(k_1) \prod_{i=2}^{L} \Qbf(k_{i-1}, k_i) \bigotimes_{i=1}^L f_{k_i}.
\end{equation*}

When $\pi$ is a probability distribution on $[K]$, $\Qbf$ a $K \times K$ transition matrix and $\fbf$ a $K$-uple of probability densities, $g^{\pi, \Qbf, \fbf, K}$ is the density of the first $L$ observations of a HMM with parameters $(\pi, \Qbf, \fbf, K)$.

For the sake of readability, we will drop the dependence in $K$ in the following and write $g^{\pi, \Qbf, \fbf}$ instead of $g^{\pi, \Qbf, \fbf, K}$. Moreover, if $\Qbf$ is irreducible with stationary distribution $\pi$, we simply write $g^{\Qbf, \fbf}$, and we write the true density $g^* := g^{\pi^*, \Qbf^*, \fbf^*}$.

\subsection{Assumptions}
\label{sec_assumptions}

Let $\Fcal$ be a subset of $\Lbf^2(\Ycal, \mu)$ and $(\Pfrak_M)_{M \in \Mcal \subset \Nbb}$ be a sequence of nested subspaces of $\Lbf^2(\Ycal, \mu)$ such that $\Pfrak_M$ has dimension $M$ for all $M \in \Mcal$ and their union is dense in $\Lbf^2(\Ycal, \mu)$. $(\Pfrak_M)_{M \in \Mcal}$ will be the subspaces on which the projections of the emission densities will be estimated.

We will need the following assumptions.

\begin{description}
\item[\textbf{[HX]}] $(X_k)_{k \geq 1}$ is a stationary ergodic Markov chain with parameters $(\pi^*, \Qbf^*)$;

\item[\textbf{[HidA]}] $\Qbf^*$ is invertible, $L \geq 3$ and the family $\fbf^*$ is linearly independent;

\item[\textbf{[HidB]}] $\Qbf^*$ is invertible, $L \geq (2 K^* + 1)((K^*)^2 - 2K^* + 2) + 1$ and the emission densities $(f^*_k)_{k \in \Xcal}$ are all distinct;

\item[\textbf{[HF]}] $\fbf^* \in \Fcal^{K^*}$, $\Fcal$ is closed under projection on $\Pfrak_M$ for all $M$ and
\begin{equation*}
\forall f \in \Fcal, \quad
\begin{cases}
\| f \|_\infty \leq \cinf \\
\| f \|_2 \leq \cquad
\end{cases}
\end{equation*}
with $\cinf$ and $\cquad$ larger than 1.
\end{description}

The ergodicity assumption in \textbf{[HX]} is completely standard in order to obtain convergence results. In this case, the initial distribution is forgotten exponentially fast, so that the HMM will essentially behave like a stationary process. In order to simplify the proofs, we assume the Markov chain to be stationary. One can check that our results are essentially the same when the initial distribution is not the stationary one.

\textbf{[HidA]} appears in spectral methods, with the hypothesis that $\pi^* > 0$ elementwise, see for instance \cite{HKZ12}. \textbf{[HidA]} and \textbf{[HidB]} also appear in identifiability issues, possibly combined with the stationarity hypothesis, see \cite{AH14} and \cite{GCR15}. Note that the condition on $L$ in \textbf{[HidB]} only involves the real order $K^*$.

Even though \textbf{[HidB]} appears less restrictive than \textbf{[HidA]} about the emission densities, it is delicate to use here. The problem lies in the condition on the number of consecutive observations $L$. For \textbf{[HidB]}, one has to take $L$ larger than an increasing function of the order, so it requires to have an \emph{a priori} upper bound on the order to choose $L$. This is less interesting than \textbf{[HidA]}, which can work without prior bound since it only requires $L=3$ for any value of the order. 

\section{Least squares estimation}
\label{sec_least_squares}

In this section, we introduce our penalized least squares estimator and study its asymptotic properties.

\subsection{Approximation spaces and estimators}

We want to estimate the density of $L$ consecutive observations $g^*$ by minimizing the quadratic loss $t \mapsto \|t - g^*\|_2^2 - \| g^* \|_2^2$. We thus take the corresponding empirical loss
\begin{equation*}
\gamma_n(t) = \| t \|_2^2 - \frac{2}{n} \sum_{s=1}^n t(Z_s)
\end{equation*}
where $Z_s = (Y_s, \dots, Y_{s+L-1})$ for an observation sequence $(Y_t)_{1 \leq t \leq n+L-1}$ of length ${n + L - 1}$ coming from a single HMM $(X_t,Y_t)_{t \geq 1}$.

Define for all $K \in \Nbb^*$, $M \in \Mcal$:
\begin{align*}
S_{K,M} &:= \{ g^{\Qbf,\fbf} , \Qbf \in \Qcal_K, \fbf \in (\Fcal \cap \Pfrak_M)^K \} \\
S_{K} &:= \{ g^{\Qbf,\fbf} , \Qbf \in \Qcal_K, \fbf \in \Fcal^K \}
\end{align*}
where $\Fcal$ and $(\Pfrak_M)_{M \in \Mcal}$ are defined in Section \ref{sec_assumptions}. In the following, we will always implicitly consider $M \in \Mcal$.

For all $K$ and $M$, we define the corresponding estimators
\begin{equation*}
\hat{g}_{K,M} = g^{\hat{\Qbf}_{K,M}, \hat{\fbf}_{K,M}}
	\in \underset{t \in S_{K,M}}{\argmin} \; \gamma_n(t)
\end{equation*} 
where we dropped the dependency in $n$ for ease of notation. Then, we select the parameters using the penalized empirical loss:
\begin{equation*}
(\hat{K}_\text{l.s.}, \hat{M}) \in \underset{K \leq n, \, M \leq n}{\argmin} \left\{ \gamma_{n}(\hat{g}_{K,M}) + \pen(n,M,K) \right\}
\end{equation*}
which leads to the estimators
\begin{align*}
\hat{g}  &:= \hat{g }_{\hat{K}_\text{l.s.}, \hat{M}} \\
\hat{\Qbf} &:= \hat{\Qbf}_{\hat{K}_\text{l.s.}, \hat{M}} \\
\hat{\fbf} &:= \hat{\fbf}_{\hat{K}_\text{l.s.}, \hat{M}}
\end{align*}

\subsection{Underestimation of the order}

Note that the distribution of the HMM remains unchanged under permutation of the hidden states. We will therefore use a pseudo-distance $\dperm$ that is invariant by permutation on the set of parameters.

We define it as follows. Let $K \geq 1$, $\pi_1, \pi_2 \in \Delta_K$, $\Qbf_1$ and $\Qbf_2$ transition matrices of size $K$, $\fbf_1, \fbf_2 \in (\Lbf^2(\Ycal,\mu))^K$. Let $\Sfrak(\Xcal)$ be the set of permutations of $\Xcal$. For all $\tau \in \Sfrak(\Xcal)$, define the swapped parameters $\tau \pi_1$, $\tau \Qbf_1$ and $\tau \fbf_1$ by
\begin{align*}
(\tau \pi_1)(k)  &:= \pi_1(\tau(k)) \\
(\tau \Qbf_1)(k,l) &:= \Qbf_1(\tau(k), \tau(l)) \\
(\tau \fbf_1)_k   &:= f_{1,\tau(k)}
\end{align*}
and finally
\begin{multline*}
\dperm((\pi_1, \Qbf_1,\fbf_1), (\pi_2, \Qbf_2,\fbf_2))
	:= \inf_{\tau \in \Sfrak(\Xcal)} \Bigg(
		\| \tau \pi_1 - \pi_2 \|_2^2 \\
		+ \| \tau \Qbf_1 - \Qbf_2 \|_F^2
		+ \sum_{k=1}^K \| (\tau \fbf_1)_k - f_{2,k} \|_2^2
	\Bigg)^{1/2}.
\end{multline*}

The following properties will be of use to prove the consistency of the order estimator, but we think it can also be of independent interest to better understand the identifiability of the model. The first one is a generalization of previous identifiability results from \cite{AH14, GCR15, dCGL15}.

\begin{prop}
\label{prop_identifiability}
Let $K \geq 1$, $\pi \in \Delta_K$ such that $\pi_k > 0$ for all $k \in \Xcal$, $\Qbf$ transition matrix of size $K$ and $\fbf \in (\Lbf^2(\Ycal, \mu))^K$ such that \textbf{[HidA]} or \textbf{[HidB]} hold for the order $K$. Then, for all $K' \geq 1$, for all $\pi' \in \Delta_{K'}$, for all transition matrix $\Qbf'$ of size $K'$ and all $\fbf' \in (\Lbf^2(\Ycal, \mu))^{K'}$, the following holds:
\begin{multline*}
\left( g^{\pi, \Qbf, \fbf} = g^{\pi', \Qbf', \fbf'} \text{ and } K' \leq K \right) \\
	\Rightarrow
	\left( K = K' \text{ and } \dperm((\pi, \Qbf, \fbf), (\pi', \Qbf', \fbf')) = 0 \right).
\end{multline*}
\end{prop}

\begin{remark}
This property does not require two assumptions that appear in \cite{AH14} and \cite{GCR15}: that $\fbf$ is a family of probability densities and that the Markov chain is stationary.

In particular, the fact that $\fbf$ may not be a family of probability densities is crucial in the proof of Corollary \ref{cor_distance_g_ast_S_KM}, which is necessary to prove the strong consistency of the estimator of the order.
\end{remark}

\begin{proof}

Assume \textbf{[HidA]}. The spectral algorithm from \cite{dCGLLC15} applied on the linear space spanned by both sets of densities allows to retrieve the order from two consecutive observations and the parameters from three consecutive observations. Their proof works when the emission densities are not probability densities and when the chain is not stationary.

Assume \textbf{[HidB]}. A careful reading of the proofs of \cite{AH14} shows that their result can be extended to general observation spaces and do not require the measures to be probabilities.
\end{proof}

The second property is the following corollary, which states that the $\Lbf^2$ distance between the actual model and the models where the order is underestimated is positive. It is worth noting that we do not need $\Fcal$ to be compact.
\begin{corollary}
\label{cor_distance_g_ast_S_KM}
Assume \textbf{[HX]}, (\textbf{[HidA]} or \textbf{[HidB]}) and \textbf{[HF]} hold. Then, for all $K < K^*$:
\begin{equation*}
d_{K}~:= \inf_{t \in S_{K}} \| t - g^* \|_2 > 0
\end{equation*}
\end{corollary}
\begin{proof}
Proof in Section \ref{sec_proof_cor_distance}.
\end{proof}

Our first theorem shows that the probability to underestimate the order decreases exponentially with the number of observations. This comes from Corollary \ref{cor_distance_g_ast_S_KM}: since the empirical criterion converges to the $\Lbf^2$ distance (plus some constant that does not depend on the model), the penalized error will eventually become larger for orders under $K^*$ than for orders over $K^*$, which means that we won't underestimate the real order. The exponential decrease rate brings to mind the one studied in \cite{gassiat2003optimal}: in both cases, the exponents involve the distance between the actual model and models with underestimated orders, as can be seen in our proof.

\begin{theorem}
\label{th_underestimation}
Assume \textbf{[HX]}, (\textbf{[HidA]} or \textbf{[HidB]}) and \textbf{[HF]} hold. There exists positive constants $\rho \equiv \rho(\cquad, \cinf, \Qbf^*, L)$ and $\beta \equiv \beta(\cquad, \cinf, \Qbf^*, (d_K)_{K < K^*}, L)$ such that the following holds.

Assume that
\begin{equation*}
\displaystyle \forall n, \; \forall M, \; \forall K, \quad \pen(n,M,K) \geq \rho (MK + K^2 - 1) \frac{\log(n)}{n},
\end{equation*}
and
\begin{equation*}
\displaystyle \forall M, \; \forall K, \quad \pen(n,M,K) \underset{n \rightarrow \infty}{\longrightarrow} 0
\end{equation*}
then there exists $n_0$ such that for all $n \geq n_0$,
\begin{equation*}
	\Pbb(\hat{K}_\text{l.s.} < K^*) \leq e^{-\beta n}.
\end{equation*}
\end{theorem}

\begin{proof}
Proof in Section \ref{sec_proof_th_underestimation}.
\end{proof}

\subsection{Overestimation of the order and consistency}

Our second theorem controls the probability to overestimate the order. It consists in overpenalizing large models so that the estimated order remains small.

We will need the following technical condition on the penalty:
\begin{condition}[{[\textbf{Hpen}]($\alpha, \rho$)}]
The penalty function $\pen$ satisfies
\begin{multline*}
\exists n_1, \; \forall n \geq n_1, \; \forall M \leq n, \; \forall K \leq n \text{ s.t. } K > K^*, \\
\pen(n,M,K) - \pen(n,M,K^*)
	\geq \rho (MK + K^2 - 1) \frac{\log(n)}{n} + \alpha \frac{\log(n)}{n},
\end{multline*}
\end{condition}

We can now state the theorem and its corollary proving the strong consistency of our estimator of the order. Note that it does not require any identifiability assumption.

\begin{theorem}
\label{th_overestimation}
Assume \textbf{[HX]} and \textbf{[HF]} hold. There exists positive constants $(\rho, \beta) \equiv (\rho, \beta)(\cquad, \cinf, \Qbf^*, L)$ such that the following holds.

Assume \textbf{[Hpen]}($\alpha, \rho$) holds for some $\alpha \geq 0$, then there exists $n_0$ such that for all $n \geq n_0$,
\begin{equation*}
\Pbb(\hat{K}_\text{l.s.} > K^*) \leq n^{-\beta \alpha}.
\end{equation*}
\end{theorem}

\begin{proof}
Proof in Section \ref{sec_proof_th_underestimation}.
\end{proof}

\begin{corollary}
\label{cor_consistency}
Assume \textbf{[HX]}, \textbf{[HF]} and (\textbf{[HidA]} or \textbf{[HidB]}) hold. There exists positive constants $(\rho, \beta) \equiv (\rho, \beta)(\cquad, \cinf, \Qbf^*, L)$ such that the following holds.

Assume that the penalty function satisfies
\begin{equation*}
\begin{cases}
\forall n, \; \forall M \leq n, \; \forall K \leq n, \quad \pen(n,M,K) \geq \rho (MK + K^2 - 1) \frac{\log(n)}{n} \\
\forall M, \; \forall K, \quad \pen(n,M,K) \underset{n \rightarrow + \infty}{\longrightarrow} 0
\end{cases}
\end{equation*}
and \textbf{[Hpen]}$(\alpha/\beta, \rho)$ holds for some $\alpha > 1$, then
\begin{equation*}
\Pbb(\hat{K}_\text{l.s.} \neq K^*) = O(n^{-\alpha}).
\end{equation*}
In particular, $\hat{K}_\text{l.s.} \longrightarrow K^*$ almost surely.
\end{corollary}

Let us comment on the condition \textbf{[Hpen]} when using a penalty of the form $\pen(n,M,K) = C (MK + K^2 - 1) \log(n)/n$ where $C$ may depend on $n$.
\begin{itemize}
\item If one has an \emph{a priori} bound on the order, i.e. if $K^* \leq K_0$ for some known $K_0$, then direct computations show that for all $\alpha$, $\rho$, there exists $C \geq 0$ depending on $K_0$ (for instance, $C = 2\rho(1+K_0^2 \vee \frac{\alpha}{\rho})$ works) such that \textbf{[Hpen]}$(\alpha, \rho)$ holds for all $K^* \leq K_0$ (instead of $K \leq n$). This means that if one has an \emph{a priori} bound $K_0$ on the order, then by taking a constant $C$ large enough and $\hat{K}_\text{l.s.} \leq K_0$, the estimator $\hat{K}_\text{l.s.} > K^*$ is almost surely consistent.

\item If one does not have an \emph{a priori} bound on $K^*$, taking a constant $C$ does not allow to get \textbf{[Hpen]}$(\alpha, \rho)$ for all possible $K^*$, which means we can't apply Corollary \ref{cor_consistency}. However, by taking $C$ as a sequence indexed by $n$ that tends to infinity, we get that for all $K^*$ and $\alpha, \rho$, \textbf{[Hpen]}$(\alpha, \rho)$ holds. This implies consistency with polynomial decrease of the probability of error, at the cost of overpenalizing.

Overpenalizing is actually necessary if one wants to satisfy \textbf{[Hpen]} for all $K^*$. This is stated in the following proposition:
\end{itemize}

\begin{prop}
\label{prop_overpenalizing}~
Let $\rho > 0$ and $\pen$ be a positive penalty such that for all $K^*$, \textbf{[Hpen]}$(0,\rho)$ holds, then there exists a sequence $(u_n)_{n \geq 1} \longrightarrow \infty$ such that for all $n \geq 1$, $M \leq n$ and $K \leq n$, $\pen(n,M,K) \geq u_n (MK + K^2 - 1) \log(n)/n$.
\end{prop}
\begin{proof}
Proof in Appendix \ref{sec_proof_prop_overpenalizing}.
\end{proof}


\subsection{Oracle inequalities}
\label{sec_oracles}

Our first result for this section is an oracle inequality on the density of $L$ consecutive observations for the least squares estimator.

\begin{theorem}
\label{th_oracle_g}
Assume \textbf{[HX]} and \textbf{[HF]} hold. Then there exists positive constants
$(n_0, \rho, A) \equiv (n_0, \rho, A)(\cquad, \cinf, \Qbf^*, L)$ such that if the penalty satisfies
\begin{equation*}
\displaystyle \forall n, \; \forall M \leq n, \; \forall K \leq n, \quad \pen(n,M,K) \geq \rho (MK + K^2 - 1) \frac{\log(n)}{n}
\end{equation*}
then for all $n \geq n_0$, for all $x > 0$, it holds with probability larger than ${1 - e^{-x}}$ that
\begin{align*}
\| \hat{g} - g^* \|_2^2 
	&\leq 4 \inf_{K \leq n, \, M \leq n} \left\{ \| g^*_{K,M} - g^* \|_2^2 + \pen(n,M,K) \right\}
		+ 4 A \frac{x}{n}.
\end{align*}
\end{theorem}
\begin{proof}
Proof in Section \ref{sec_proof_th_oracle}.
\end{proof}

\begin{remark}
The constant $4$ before the infimum can be replaced by any constant $\kappa > 1$, at the cost of changing the constants $n_0$, $\rho$ and $A$.
\end{remark}

We would like to deduce an oracle inequality on the parameters of the HMM from this result.
Using Cauchy-Schwarz inequality, it is easy to upper bound the error on the density $g^*$ by the error on the parameters: for all probability distributions $\pi_1$ and $\pi_2$ on $[K]$, for all transition matrices $\Qbf_1$ and $\Qbf_2$ of size $K$ and for all $\fbf_1, \fbf_2 \in \Fcal^K$,
\begin{equation}
\label{eq_upperbound}
\| g^{\pi_1, \Qbf_1,\fbf_1} - g^{\pi_2, \Qbf_2,\fbf_2} \|_2
	\leq
	\cquad^L \sqrt{L K} \dperm ((\pi_1, \Qbf_1,\fbf_1), (\pi_2, \Qbf_2,\fbf_2))
\end{equation}
as soon as \textbf{[HF]} holds. The proof of this equation is detailed in Section \ref{sec_proof_eq_upperbound}.

Thus, all we need to deduce an oracle inequality on the parameters is to lower bound the error on $g^*$ by the error on the parameters.
Let $\Cfrak \subset \Rbb^{K^*} \times \Rbb^{K^* \times K^*} \times \Rbb^{K^* \times K^*}$ be the set of parameters $(p,q,A)$ such that
\begin{equation}
\label{eq_conditions_det_invertible}
\begin{cases}
\forall i \in \Xcal, \quad \sum_{j \in \Xcal} q(i,j) = 0 \\
\forall j \in \Xcal, \quad \sum_{i \in \Xcal} A(i,j) = 0
\end{cases}
\end{equation}
Note that $\Cfrak$ can be identified with the set
\begin{align*}
\Cfrak_\text{red} :=& \{ ((p_i)_{i \geq 2},(q(i,j))_{i, j \geq 2},(A(i,j))_{i \geq 2, j}) \; | \; (p,q,A) \in \Cfrak \} \\
	=& \, \Rbb^{K^*-1} \times \Rbb^{K^* \times (K^*-1)} \times \Rbb^{K^* \times (K^*-1)}
\end{align*}
These assumptions are natural since they are necessary (but not sufficient) to ensure that if $(p,q,A) \in \Cfrak$ and $\pi$ is a probability distribution, $\Qbf$ a transition matrix and $\fbf$ a vector of probability densities, then $\pi+p$ is also a probability distribution, $\Qbf + q$ a transition matrix and $\fbf + A \fbf$ a vector of probability densities.

The first step in order to get a lower bound along the same lines as equation (\ref{eq_upperbound}) is to control the behaviour of the difference near the true parameters, which comes down to proving that the quadratic form $M$ derived from the second-order expansion of
\begin{align*}
\Nfrak : (p,q,A) \in \Rbb^{K^*} \times \Rbb^{K^* \times K^*} \times \Rbb^{K^* \times K^*}
	\longmapsto \| g^{\pi + p, \Qbf + q, \fbf + A \fbf} - g^{\pi, \Qbf, \fbf} \|_2^2
\end{align*}
is positive definite on $\Cfrak$ for $(\pi, \Qbf, \fbf) = (\pi^*, \Qbf^*, \fbf^*)$. One can write the coefficients of the matrix of this quadratic form as polynomials in the coefficients of $\pi$, $\Qbf$ and of the Gram matrix $G(\fbf) := (\langle f_i, f_j \rangle)_{i,j \in \Xcal}$. However, this matrix may not be invertible: one has to consider its restriction to the space $\Cfrak$, which is equivalent to considering the quadratic form $M_\Cfrak$ defined on $\Cfrak_\text{red}$ by the second-order expansion of $x \in \Cfrak_\text{red} \longmapsto \Nfrak(I_\Cfrak(x))$ where $I_\Cfrak$ is the natural linear injection from $\Cfrak_\text{red}$ to $\Cfrak$ (note that $I_\Cfrak$ is bijective and bicontinous under \textbf{[HF]}). Since the quadratic form $M_\Cfrak$ is always nonnegative, we only need its determinant to be non zero in order for the quadratic form $M$ to be positive definite on $\Cfrak$.

Thus, let $H$ be determinant of the matrix of this quadratic form. $H$ is also a polynomial in the coefficients of $\pi$, $\Qbf$ and $G(\fbf)$. The following lemma shows that there exists some parameters $\pi, \Qbf$ and $\fbf$ satisfying the conditions for which $H$ is not zero.

\begin{lemma}
\label{lemma_H_nonzeropolynomial}
There exists some parameters $(\pi, \Qbf, \fbf)$ satisfying the conditions \textbf{[HX]} and \textbf{[HidA]} such that $H(\pi, \Qbf, G(\fbf)) \neq 0$.
\end{lemma}
\begin{proof}
Proof in Section \ref{sec_proof_H_nonzeropolynomial}.
\end{proof}

What should be retained from this lemma is that $H$ is a polynomial which is not identically zero on the set of parameters satisfying the identifiability conditions. This means that one can generically assume it to be different from zero, which corresponds to the assumption
\begin{description}
\item[\textbf{[Hdet]}]
$H(\pi^*, \Qbf^*, G(\fbf^*)) \neq 0$.
\end{description}

Since we assumed $\pi^*$ to be the stationary distribution of $\Qbf^*$, its coefficients\---and by extension $H$\---can be expressed as a rational function of the coefficients of $\Qbf^*$. Taking $H_1$ as the numerator of the rational function deduced from $H$, one gets another polynomial in the coefficients of $\Qbf^*$ and $G(\fbf^*)$ which is also non-zero. Thus, the following assumption\---which we will need to lower bound the error on the density $g^*$ by the error on the parameters\---is generically satisfied.
\begin{description}
\item[\textbf{[HdetStat]}]
$H_1(\Qbf^*, G(\fbf^*)) \neq 0$.
\end{description}

Note that \textbf{[Hdet]} and \textbf{[HdetStat]} are equivalent under the assumption \textbf{[HX]}.

\begin{theorem}
\label{th_lowerbound}
Assume \textbf{[HidA]} and \textbf{[Hdet]} hold. Then there exists a positive constant $c(\pi^*, \Qbf^*, \fbf^*)$ such that for all $\pi \in \Delta_{K^*}$, for all transition matrix $\Qbf$ of size $K^*$ and for all $\hbf \in \Fcal^{K^*}$ such that $\int h_i d\mu = 1$ for all $i \in [K^*]$,
\begin{equation*}
\| g^{\pi, \Qbf, \hbf} - g^{\pi^*, \Qbf^*, \fbf^*} \|_2^2
	\geq c(\pi^*, \Qbf^*, \fbf^*) \; \dperm((\pi, \Qbf, \hbf), (\pi^*, \Qbf^*, \fbf^*))^2.
\end{equation*}
\end{theorem}
\begin{proof}
Proof in Section \ref{sec_proof_lowerbound}.
\end{proof}

The following theorem is a direct consequence of the above results. It provides an oracle inequality on the parameters conditionally to the fact that the order has been correctly estimated.

\begin{theorem}
\label{th_oracle_param}
Assume \textbf{[HX]}, \textbf{[HidA]}, \textbf{[HF]} and \textbf{[Hdet]} hold. Also assume that for all $f \in \Fcal$, $\int f d\mu = 1$.

Then there exists positive constants $(n_0, \rho, A) \equiv (n_0, \rho, A)(\cquad, \cinf, \Qbf^*, L)$ such that if the penalty satisfies
\begin{equation*}
\displaystyle \forall n, \; \forall M \leq n, \; \forall K \leq n, \quad \pen(n,M,K) \geq \rho (MK + K^2 - 1) \frac{\log(n)}{n}
\end{equation*}
then for all $n \geq n_0$, for all $x > 0$, conditionally to  $\{ \hat{K}_\text{l.s.} = K^* \}$, with probability larger than ${1 - e^{-x}}$:
\begin{multline*}
\dperm((\hat{\pi}, \hat{\Qbf}, \hat{\fbf}), (\pi^*, \Qbf^*, \fbf^*))
	\leq \frac{4 \cquad^L \sqrt{L K^*}}{c(\Qbf^*, \fbf^*)} \times \\
	\Bigg[ \inf_{M \leq n} \Bigg\{ \sum_{k=1}^{K^*} \| f^*_{M, k} - f^*_k \|_2^2
	+ \pen(n,M,K^*) \Bigg\}
		+ A \frac{x}{n} \Bigg],
\end{multline*}
where $f^*_{M, k}$ is the projection of $f^*_k$ on $\Pfrak_M$.
\end{theorem}

It is now possible to get the convergence rate of the estimators of the parameters. In order to take the event where $\hat{K}_\text{l.s.} \neq K^*$ into account, we agree that the distance between the parameters of two HMMs with different orders is bounded by some constant $C_\text{err}$. Note that $C_\text{err}$ could even be taken as a power of $n$ without changing anything to our result.
\begin{corollary}
\label{cor_oracle_even_when_mistaken}
Assume \textbf{[HX]}, \textbf{[HidA]}, \textbf{[HF]} and \textbf{[Hdet]} hold.  Also assume that for all $f \in \Fcal$, $\int f d\mu = 1$, and that the penalty satisfies
\begin{equation*}
\displaystyle \forall n, \; \forall M \leq n, \; \forall K \leq n, \quad \pen(n,M,K) = (MK + K^2 - 1) \frac{\log(n)^2}{n}
\end{equation*}
Then there exists a positive constant $A \equiv A(\cquad, \cinf, \Qbf^*, L)$ such that for all $\beta > 1$, there exists a positive constant $n_0 \equiv n_0(\cquad, \cinf, \Qbf^*, L, \beta)$ such that for all $n \geq n_0$ and for all $C_\text{err} > 0$,
\begin{multline*}
\Ebb\left[ \one_{\hat{K} \neq K^*} C_\text{err} +  \one_{\hat{K} = K^*} \dperm((\hat{\pi}, \hat{\Qbf}, \hat{\fbf}), (\pi^*, \Qbf^*, \fbf^*)) \right]
	\leq \frac{4 \cquad^L \sqrt{L K^*}}{c(\Qbf^*, \fbf^*)} \times \\
	\inf_{M \leq n} \Bigg\{ \sum_{k=1}^{K^*} \| f^*_{M, k} - f^*_k \|_2^2
	+ \pen(n,M,K^*) \Bigg\}
		+ \frac{A}{c(\Qbf^*, \fbf^*) \, n}
		+ \frac{C_\text{err}}{n^\beta},
\end{multline*}
and $\Pbb(\hat{K}_\text{l.s.} \neq K^*) = O(n^{-\beta}).$
\end{corollary}
Let us discuss what this corollary implies. The approximation error $\sum_{k=1}^{K^*} \| f^*_{M, k} - f^*_k \|_2^2$ can be bounded in a standard way by $O(M^{-2s/D})$ where $s > 0$ is the regularity of the emission densities, see for instance \cite{devore1993constructive}. One can obtain a trade-off between approximation error and penalty by choosing $M \approx (n/\log(n)^2)^{D/(2s+D)}$, which leads to the optimal rate of convergence $(n/\log(n)^2)^{-2s/(2s+D)}$, up to a logarithmic factor. This shows that our estimators are adaptive, quasi-rate minimax and converge almost surely to the right number of states, all at the same time.

\section{Spectral estimation}
\label{sec_spectral_theory}

In this section, we introduce our spectral order estimator. We will assume \textbf{[HX]} and \textbf{[HidA]} hold.

The idea of this method is to use the matrix containing the coordinates of the density of two consecutive observations in an orthonormal basis. Take $M \in \Mcal$ and let $\Phi_M = (\varphi_1^{(M)}, \dots, \varphi_M^{(M)})$ be an orthonormal basis of $\Pfrak_M$. For ease of notation, we will drop the dependency in $M$ and write $\varphi_a$ instead of $\varphi_a^{(M)}$. Let us introduce the matrice $\Nbf_M$ and its empirical estimator, defined by
\begin{align*}
\forall a,b \in [M], \quad \Nbf_M(a,b) &:= \Ebb [ \varphi_{a}(Y_1)\varphi_{b}(Y_2) ], \\
\forall a,b \in [M], \quad \hat{\Nbf}_M(a,b) &:=\frac{1}{n} \sum_{s=1}^{n}\varphi_{a}(Y_{s})\varphi_{b}( Y_{s+1}).
\end{align*}
$\Nbf_M$ contains the coordinates of the density of $(Y_1, Y_2)$ with respect to $\mu^{\otimes 2}$ on the basis $\Phi_M$. It holds that
\begin{equation}
\Nbf_M = \Obf_M \text{Diag}(\pi^*) \Qbf^* \Obf_M^\top,
\label{eq_NM}
\end{equation}
with $\Obf_M$ the coordinates of the emission densities on the orthonormal basis:
\begin{equation*}
\forall m \in [M], \; \forall k \in \Xcal, \quad \Obf_M(m,k) := \Ebb [ \varphi_{m}(Y_1) | X_1 = k ] = \int \varphi_m f^*_k d\mu.
\end{equation*}
When the emission densities are linearly independent, $\Obf_M$ has full rank for $M$ large enough.

The key remark for our method is that $\Nbf_M$ contains explicit information about the order of the HMM, as stated in the following lemma:
\begin{lemma}
\label{lemma_rankNM}
There exists $M_0 \equiv M_0(\Qbf^*, \Phi_M, \fbf^*)$ such that for all $M \geq M_0$, $\Nbf_M$ has rank $K^*$.
\end{lemma}

In the following, we will assume $M \geq M_0$ for $M_0$ given by this lemma.

In practice, one only has access to the matrix $\hat{\Nbf}_M$, which can be seen as a noisy version of $\Nbf_M$. In particular, there is no reason for it to have only $K^*$ nonzero singular values. On the contrary, the spectrum becomes noisy, and when some singular values of $\Nbf_M$ are too small, they can be masked by this noise. As seen in equation (\ref{eq_NM}), this can occur when $\Qbf^*$ or $\Obf_M$ are close to not having full rank, which means for $\Obf_M$ that the emission densities are almost linearly dependent.

Denote by $\sigma_1(A) \geq \sigma_2(A) \geq \dots$ the singular values of the matrix $A$. We can now state the theorem proving the consistency of the spectral order estimator:

\begin{theorem}
\label{th_spectr_consistent}
Let $\hat{K}_\text{sp.}(C) = \# \{ i \; | \; \sigma_i(\hat{\Nbf}_M) > C \sqrt{\log(n) / n} \}$.

There exists $C_0 \equiv C_0(\Qbf^*, \Phi_M)$ and $n_0 \equiv n_0(\Qbf^*, \Phi_M, \Obf_M^*)$ such that for all $C \geq C_0$ and $n \geq n_0 C^2 (1 + \log(C))$,
\begin{equation*}
\Pbb(\hat{K}_\text{sp.}(C) \neq K^*) \leq n^{-2}
\end{equation*}
so that $\hat{K}_\text{sp.}(C) \longrightarrow K^*$ almost surely.
\end{theorem}

\begin{remark}
It is possible to take $M \longrightarrow \infty$, $n_0$ constant and $C_0$ depending on $M$ in an explicit way as long as $M$ grows slowly enough, that is ${\eta_2(\Phi_M) \leq \text{cst} \cdot \sqrt{n/\log(n)}}$ and $C_0 = \text{cst} \cdot \eta_2(\Phi_M)$ where $\eta_2(\Phi_M)$ is defined in Lemma \ref{lemma_deviations_NM}.
\end{remark}

\begin{proof}
The following result from appendix E of \cite{dCGLLC15} allows to control the difference between the spectra of $\Nbf_M$ and $\hat{\Nbf}_M$.

\begin{lemma}
\label{lemma_deviations_NM}
There exists some constant $\Ccal_*$ depending only on $\Qbf^*$ such that for any positive $u$, $M$ and $n$,
\begin{equation*}
\Pbb\left[
	\| \Nbf_M - \hat{\Nbf}_M \|_F
		\geq \frac{\eta_2(\Phi_M) \Ccal_*}{\sqrt{n}} (1 + u)
\right]
	\leq e^{-u^2}
\end{equation*}
where
\begin{equation*}
\eta_2^2(\Phi_M) = \sup_{y,y' \in \Ycal^2} \, \sum_{a,b=1}^M \, (\varphi_a(y_1) \varphi_b(y_2) - \varphi_a(y'_1) \varphi_b(y'_2))^2.
\end{equation*}
\end{lemma}
In particular, taking $u = \sqrt{2 \log(n)}$ and assuming $u > 1$ and $n \geq 2$, one has with probability $1 - n^{-2}$ that
\begin{equation*}
\sigma_1(\Nbf_M - \hat{\Nbf}_M) \leq C \sqrt{\frac{\log(n)}{n}}
\end{equation*}
for all $C \geq C_0 := 2 \sqrt{2} \eta_2(\Phi_M) \Ccal_*$, using that for any matrix $A$, one has $\sigma_1(A) \leq \| A \|_F$.

Let $C \geq C_0$. We will need Weyl's inequality (a proof may be found in \cite{stewart1990matrix} for instance):
\begin{lemma}[Weyl's inequality]
Let $A,B$ be $p \times q$ matrices with $p \geq q$, then for all $i = 1 , \dots, q$,
\begin{equation*}
| \sigma_i(A+B) - \sigma_i(A) | \leq \sigma_1(B).
\end{equation*}
\end{lemma}

Using this inequality, one gets that with probability at least $1 - n^{-2}$, for all $1 \leq i \leq K^*$, $\sigma_i(\hat{\Nbf}_M) > \sigma_{K^*}(\Nbf_M) - C \sqrt{\log(n)/n}$ and for all $i > K^*$, $\sigma_i(\hat{\Nbf}_M) < C \sqrt{\log(n)/n}$.

In particular, if $2 C \sqrt{\log(n)/n} < \sigma_{K^*}(\Nbf_M)$, then with probability at least $1 - n^{-2}$, the order is exactly the number of singular values of $\hat{\Nbf}_M$ which are larger than $C \sqrt{\log(n)/n}$. Finally, observe that under the condition $n \geq n_0 C^2 (1 + \log(C))$,
\begin{align*}
C \sqrt{\frac{\log(n)}{n}} \leq& \sqrt{\frac{2 \log(C) + \log(1 + \log(C))}{n_0 (1 + \log(C))}} \\
	\leq& \sqrt{\frac{3}{n_0}} \sqrt{\frac{\log(C)}{1 + \log(C)}},
\end{align*}
since one can assume without loss of generality that $C_0 \geq 1$. By taking $n_0 = 12/ \sigma_{K^*}(\Nbf_M)^2$, this concludes the proof.
\end{proof}

\section{Numerical experiments}
\label{sec_numerical_experiments}

In this section, we show the results of our estimators on simulated data. The simulation parameters are introduced in Section \ref{sec_simulation_parameters}. We show the numerical results and discuss their ability to select the right order in practice in Section \ref{sec_simulation_results}, and we present the data-driven methods and heuristics  we used for the numerical implementation in Section \ref{sec_simulation_heuristics}.

\subsection{Simulation parameters}
\label{sec_simulation_parameters}

We will consider $\Ycal = [0,1]$ with $\mu$ being the Lebesgue measure. We will use a trigonometric basis on $\Lbf^2([0,1])$ to generate the approximation spaces $(\Pfrak_M)_M$. More precisely, define
\begin{align*}
\varphi_0(t) &= 1 \\
\varphi_a(t) &= \sqrt{2} \cos(\pi a t) 
\end{align*}
for all $t \in [0,1]$ and $a \in \Nbb^*$. We take $\Pfrak_M = \Span(\{ \varphi_a \; | \; 0 \leq a < M \})$ the spaces induced by the trigonometric basis.

\begin{remark}
Taking the same vectors in all bases is not mandatory to ensure theoretical consistency, but in practice it allows us to take an additional initial point for the minimization step and improves the stability of the algorithm (see Step 1 below).

\end{remark}

We will assume $\fbf^*$ to be linearly independent, so that one only needs $L = 3$ observations to recover the parameters of the HMM.

In order to assess the performances of the different procedures, we generate $n$ observations of a HMM of order 3 for several values of $n$, using the following parameters:
\begin{itemize}
\item Emission distributions: Beta distributions with two possible sets of parameters: [$(1.5; 5)$, $(7; 2)$ and $(6; 6)$] or [$(2; 5)$, $(4; 2)$ and $(4; 4)$];
\item Markov chain parameters:
\begin{align*}
\Qbf^* &=
\left(\begin{array}{c c c}
0.8 & 0.1 & 0.1\\
0.2 & 0.7 & 0.1\\
0.07 & 0.13 & 0.8
\end{array}\right), \\
\pi^* &= (\frac{47}{120} \quad \frac{11}{40} \quad \frac{1}{3}) \\
	&\approx (0.3917 \quad 0.2750 \quad 0.3333).
\end{align*}
\end{itemize}

Finally, we take $M_{\text{max}} = 50$ and $K_{\text{max}} = 5$ the maximum values of $M$ and $K$ for which we will compute the estimators.

The simulation codes are available in MATLAB at \url{https://www.normalesup.org/~llehericy/HMM_order_simfiles/}.

\subsection{Numerical results}
\label{sec_simulation_results}

\begin{figure}[!h]
\centering
\begin{subfigure}[b]{0.49\textwidth}
\centering
\begin{tabular}{|c|c|c|}
\hline
$n$&$\Pbb(\hat{K}_{\text{l.s.}} = K^*)$&$\Pbb(\hat{K}_{\text{sp.}} = K^*)$ \\ \hline \hline
$999 $   &0.2 &0 \\ \hline
$3\,000 $&1   &0 \\ \hline
$9\,999$ &1   &1 \\ \hline
$19\,998$&1   &1 \\ \hline
\end{tabular}
\caption{Beta parameters $(1.5; 5)$, $(7; 2)$ and $(6; 6)$.}
\end{subfigure}
\begin{subfigure}[b]{0.49\textwidth}
\centering
\begin{tabular}{|c|c|c|}
\hline
$n$&$\Pbb(\hat{K}_{\text{l.s.}} = K^*)$&$\Pbb(\hat{K}_{\text{sp.}} = K^*)$ \\ \hline \hline
$7\,500 $&0.3 &0 \\ \hline
$19\,998$&0.9 &0 \\ \hline
$30\,000$&1   &0 \\ \hline
$49\,998$&1   &0.1 \\ \hline
\end{tabular}
\caption{Beta parameters $(2; 5)$, $(4; 2)$ and $(4; 4)$}
\end{subfigure}
\caption{Probability to select the right order for the two methods ($\hat{K}_{\text{l.s.}}$ for the least squares method and $\hat{K}_{\text{sp.}}$ for the spectral method). 10 simulations have been done for each $n$. Parameters for spectral selection are $M=40$, $M_\text{reg}=35$ and $\tau = 1.5$ (see Section \ref{sec_spectral_heuristics} for the definition of these parameters).}
\label{fig_synth}
\end{figure}

Figure \ref{fig_synth} summarizes the results of both procedures. Both select the right order as soon as the number of observations is sufficient.

The spectral method is easily put in pratice and runs extremely fast. It doesn't need a time-consuming contrast minimization step or an initial point. However, the thresholding of the singular values is a delicate issue, and if the order is incorrect, then the theoretical results about the spectral estimators of the parameters don't hold and this method may behave poorly.

The performances of the least squares method are much better (see Figure \ref{fig_synth} for comparing the order estimators and \cite{dCGL15} for comparing the emission densities estimators). In addition, the model selection step is easy to handle and gives an estimator of the order that we proved to be consistent, estimators of the HMM parameters that we proved to be quasi-rate minimax and a way to check whether the model fits the data well (see Section \ref{sec_least_squares_heuristics}), all at the same time. However, the minimization of the (non-convex) empirical contrast is a time-consuming step, especially for large samples and large models.

Choosing the right method is thus a question of computational power and amount of available data. For small datasets where one wants to get accurate results, the least squares method is best. Conversely, on large datasets and large models, the spectral method is a good choice in order to obtain many estimators in a reasonable amount of time.

\subsection{Practical implementation}
\label{sec_simulation_heuristics}

\subsubsection{Least squares method}
\label{sec_least_squares_heuristics}

The first issue that one encounters when trying to minimize the least squares criterion $\gamma_n$ is that it is not convex. Several algorithms have been proposed to overcome this difficulty. We chose to use CMA-ES (for Covariance Matrix Adaptation Estimation Strategy, see \cite{Han06}) in order to find a minimizer. This estimator is easy to use and works well in many situations, but\---like all approximate minimization algorithms\---it requires a good initial point since it might otherwise remain stuck in local minima.

One part of our method consists in using previous estimates as initial points for further steps to counter this phenomenon, since it is likely that this way the estimators stay near the real minimizer. Our practical algorithm is the following:
\begin{enumerate}
\item Minimize $\gamma_n$ on each model, for $M \leq M_{\text{max}}$ and $K \leq K_{\text{max}}$. We take several initial points for model $(K,M)$ according to the following cases:
	\begin{itemize}
	\item $K = 1$. Use a HMM with a single state and a uniform emission distribution.
	\item $K > 1$. Take the estimator from model $(K-1, M)$. For each hidden state of the corresponding HMM, use the model where this state is duplicated. More precisely, the Markov chain $\tilde{X}$ where state $I$ is duplicated is obtained by replacing the state $I$ from chain $X$ with two states $I_1$ and $I_2$ such that for each state $S \neq I_1, I_2$,
	\begin{align*}
		\Pbb(\tilde{X}_{t+1} = I_1 \; | \; \tilde{X}_t = S) &= 
		\frac{1}{2} \Pbb(X_{t+1} = I \; | \; X_t = S) \\&= 
		\Pbb(\tilde{X}_{t+1} = I_2 \; | \; \tilde{X}_t = S) \\
		\Pbb(\tilde{X}_{t+1} = S \; | \; \tilde{X}_t = I_1) &= 
		\Pbb(X_{t+1} = S \; | \; X_t = I) \\&= 
		\Pbb(\tilde{X}_{t+1} = S \; | \; \tilde{X}_t = I_2)
	\end{align*}
	and
	\begin{align*}
	\Pbb(\tilde{X}_{t+1} = I_2 \; | \; \tilde{X}_t = I_1) &= 
		\frac{1}{2} \Pbb(X_{t+1} = I \; | \; X_t = I) \\&= 
		\Pbb(\tilde{X}_{t+1} = I_1 \; | \; \tilde{X}_t = I_2)
	\end{align*}
	\item $M > 1$. Use estimator from model $(K, M-1)$ with the $M$-th coordinate of each emission density set to zero. This is only interesting if all $\Pfrak_M$ are spanned by the first $M$ vectors of a given orthonormal basis, like for trigonometric spaces.
	\end{itemize}
	Then, after minimization from each one of these initial points, take the estimator that minimizes $\gamma_n$.
\item Tune the parameter $\rho$ of the penalty with the slope heuristics or the dimension jump method (see below) and select $\hat{M}$ and $\hat{K}$.
\item Return the estimator for $M = \hat{M}$ and $K = \hat{K}$.
\end{enumerate}

This iterative initialization procedure relies on the heuristics that when the order is underestimated, then several states are "merged" together. Duplicating a merged state will allow to separate them effectively while still taking advantage of the computations done up to now. It is meant to avoid having to recalculate all states at the same time (which could get us stuck in sub-optimal local minima) when the best solution is likely to be a small modification of the previous estimator. In addition, when the order is overestimated, it allows to make sure the empirical criterion is indeed decreasing with the dimension of the model by giving an estimator that performs at least as well as those from smaller models. This makes our method robust to an overestimation of the order.

The last practical issue is a very common one in the model selection setting: the constant $\rho$ of the penalty is unknown and has to be estimated before one can select the right model. Several data-driven estimators have been proposed to circumvent this difficulty, for instance dimension jump heuristics, slope heuristics, bootstrap or cross validation. We focus on the first two, which have several advantages in our setting. First, they are easy to use, are proved to be theoretically valid in many settings and work well in a wide range of applications (see for instance \cite{BMM12a} and references therein). Secondly, they take advantage of the structure of our problem and both give a qualitative way to check whether the choice of penalty is valid or not, and by extension whether the model is misspecified or not.

\paragraph{Dimension jump heuristics}

In this paragraph, we study the selected parameters
\begin{equation*}
\rho \longmapsto (\hat{M}(\rho), \hat{K}(\rho)) \in \argmin \{ \gamma_{n}(\hat{g}_{K,M}) + \rho \, \pen_{\text{shape}}(n,M,K) \}
\end{equation*}
and the selected complexity
\begin{equation*}
\rho \longmapsto \text{Comp}(\rho) = \hat{M}(\rho) \hat{K}(\rho) + \hat{K}(\rho))[\hat{K}(\rho)) - 1]
\end{equation*}
with $\pen_{\text{shape}}(n,M,K) = (MK+K^2-1)\log(n)/n$.

Assume that there exists $\kappa$ such that $\kappa \, \pen_{\text{shape}}$ is a \emph{minimal penalty}, that is a penalty such that as $n$ tends to infinity, for all $\rho > \kappa$, the size of the model chosen for penalty $\rho \, \pen_{\text{shape}}$ remains small in some sense and for all $\rho < \kappa$, the size of the model becomes huge. Then, for $n$ large enough, this will appear on the graph of the selected model complexity as a ``dimension jump'': around some constant $\rho_\text{jump}$, the complexity will abruptly drop from large models to small models. This is clearly the case in Figure \ref{fig_dimJump_complexity}. Figure \ref{fig_dimJump_KM} shows the behaviour of $\hat{M}$ and $\hat{K}$ with $\rho$. A dimension jump also occurs with these functions. It is most visible for $\hat{M}$.

Finally, once the dimension jump location $\rho_{\text{jump}}$ has been estimated, we take $\hat{\rho} = 2 \rho_{\text{jump}}$ to select the final parameters.

It is worth noting that this jump method also gives a qualitative way to check whether the choice of parameters is sensible: if no clear jump can be identified, then either one didn't consider enough models to make the jump clear, or the penalty isn't the right one, or the model cannot approximate the data distribution well.

\begin{figure}[!h]
\centering
\begin{subfigure}[b]{0.49\textwidth}
\centering
\includegraphics[scale=0.8]{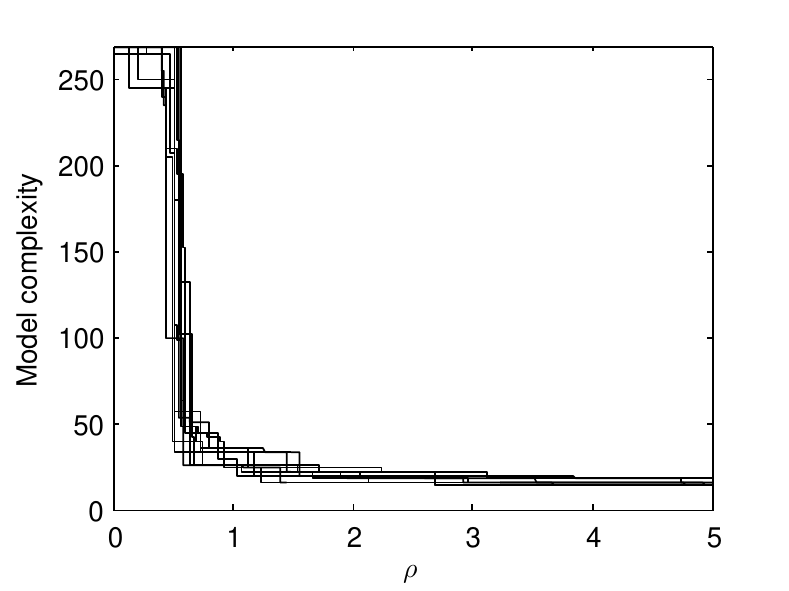}
\caption{$n = 7,500$.}
\end{subfigure}
\begin{subfigure}[b]{0.49\textwidth}
\centering
\includegraphics[scale=0.8]{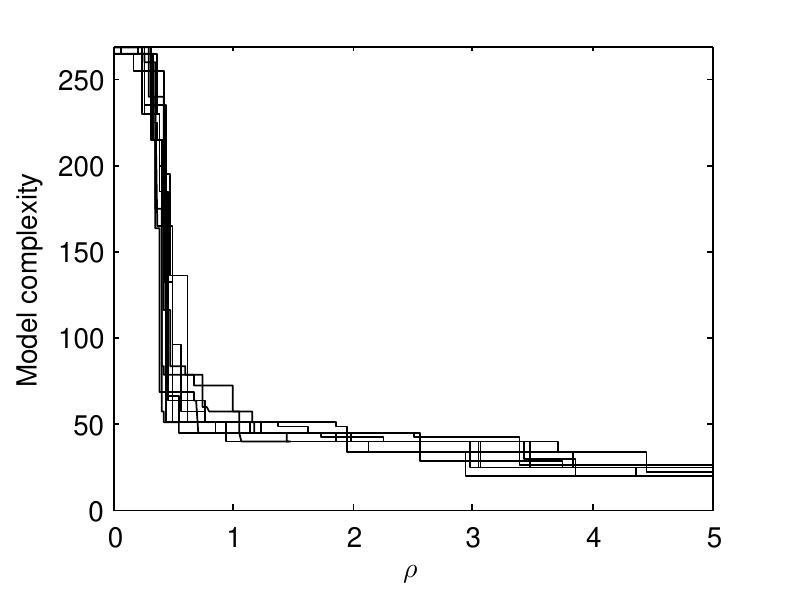}
\caption{$n = 49,998$.}
\end{subfigure}
\caption{Graph of $\rho \mapsto \text{Comp}(\rho)$ for 10 sets of $n$ consecutive observations. Here, the parameters of the Beta distribution are $(2; 5)$, $(4; 2)$ and $(4; 4)$.}
\label{fig_dimJump_complexity}
\end{figure}

\begin{figure}[!h]
\centering
\begin{subfigure}[b]{0.49\textwidth}
\centering
\includegraphics[scale=0.8]{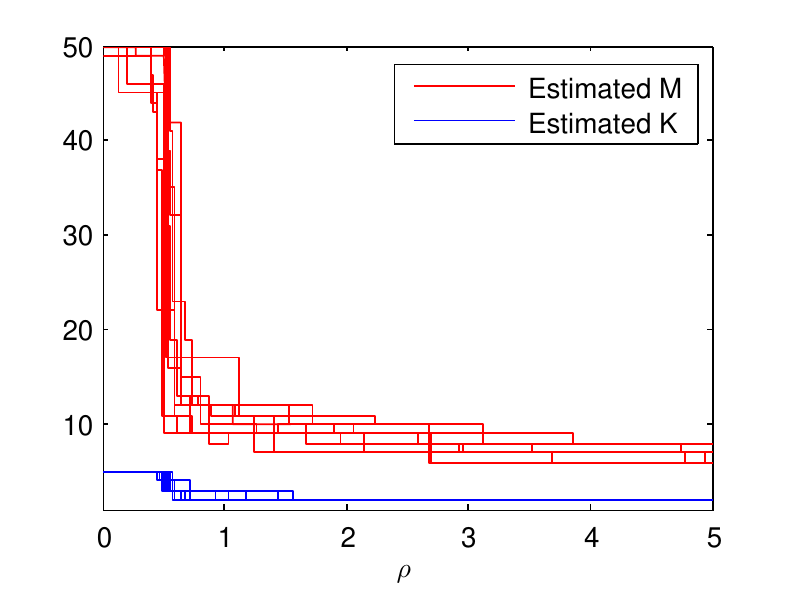}
\caption{$n = 7,500$.}
\end{subfigure}
\begin{subfigure}[b]{0.49\textwidth}
\centering
\includegraphics[scale=0.8]{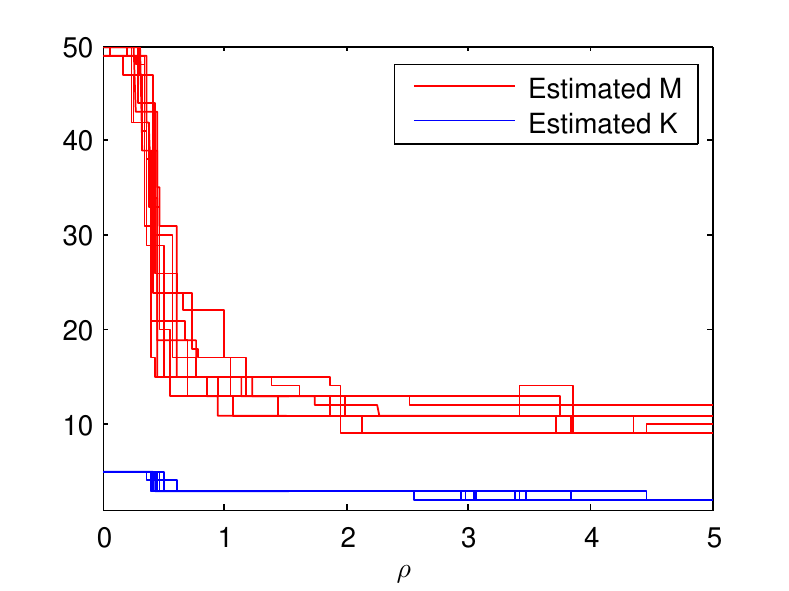}
\caption{$n = 49,998$.}
\end{subfigure}
\caption{Graph of $\rho \mapsto \hat{M}(\rho)$ and $\rho \mapsto \hat{K}(\rho)$ for 10 sets of $n$ consecutive observations. Here, the parameters of the Beta distribution are $(2; 5)$, $(4; 2)$ and $(4; 4)$.}
\label{fig_dimJump_KM}
\end{figure}

\paragraph{Slope heuristics}

This heuristics relies on the fact that when $\pen_{\text{shape}}$ is a minimal penalty, then the empirical contrast function is expected to behave like $\rho_\text{min} \pen_{\text{shape}}$ for large models and for some constant $\rho_\text{min}$. This gives both a way to calibrate the constant of the penalty and to check if the chosen penalty has the right shape (see \cite{BMM12a}). The final penalty is then taken as $2 \hat{\rho}_\text{min} \pen_{\text{shape}}$.

Figure \ref{slope_N49998} shows the graph of the empirical contrast depending on $\pen_\text{shape}$. The slope heuristics works well in this situation, suggesting that our penalty has the right shape.

\begin{figure}[!h]
\centering
\includegraphics[scale=1]{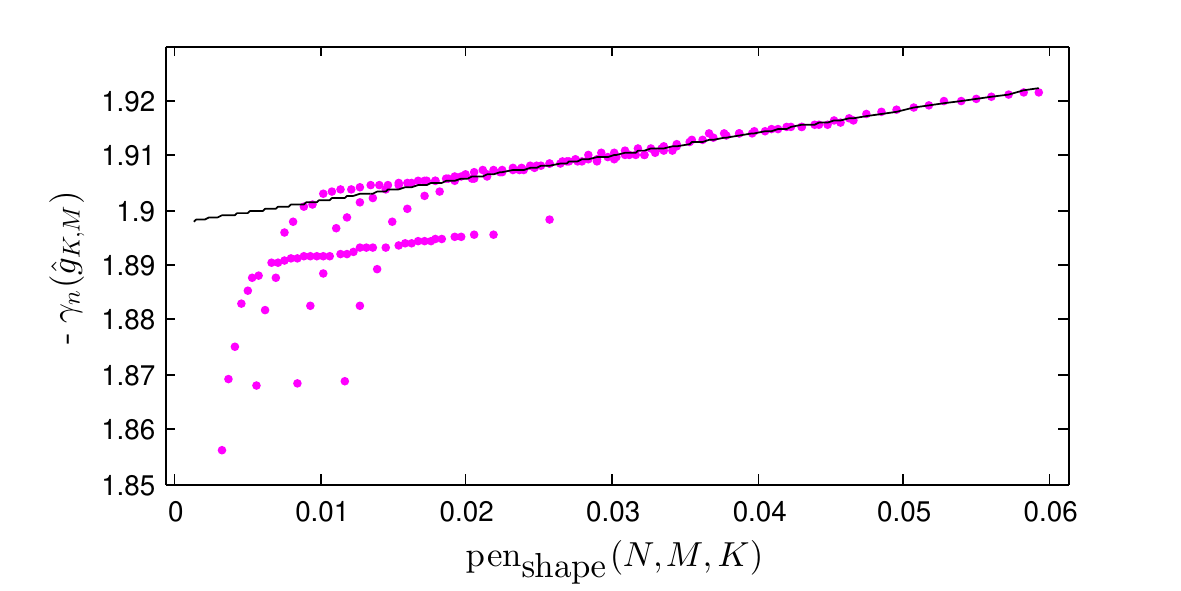}
\caption{Empirical constrast and calibrated penalty for $n=49,998$. Here, the parameters of the Beta distribution are $(2; 5)$, $(4; 2)$ and $(4; 4)$.}
\label{slope_N49998}
\end{figure}

\subsubsection{Spectral method}
\label{sec_spectral_heuristics}

The idea of the spectral order estimation is to recover the rank of the matrix $\Nbf_M$. However, this is not always possible: if one singular value of $\Nbf_M$ is smaller than the noise (which is the case when $\Obf_M$ is close from not being invertible, i.e. when the emission densities are close from being linearly dependent, and when there are only few observations), then this method will not be able to ``see'' the corresponding hidden state.

Figure \ref{fig_spectrum_N}\---and in particular Figure \ref{fig_spectrum_N_19998A}\---illustrates this problem: the third singular value is smaller than several noisy singular values, which means it won't be possible to recover it. Even if one knows the right order, the fact that the singular value is smaller than the noise can make it impossible for spectral methods to recover the true parameters. Figure \ref{fig_estimators_N19998} shows the result when trying to estimate the densities in the situation of Figure \ref{fig_spectrum_N_19998A}: when the singular value is drowned by the noise, the output of the spectral estimator is aberrant. Notice that it is not a fatality: in the same situation, the least squares method manages to give sensible estimators of the emission densities. This is an intrinsic limitation of the spectral method.

\begin{figure}[!h]
\centering
\begin{subfigure}[b]{0.49\textwidth}
\centering
\includegraphics[scale=0.8]{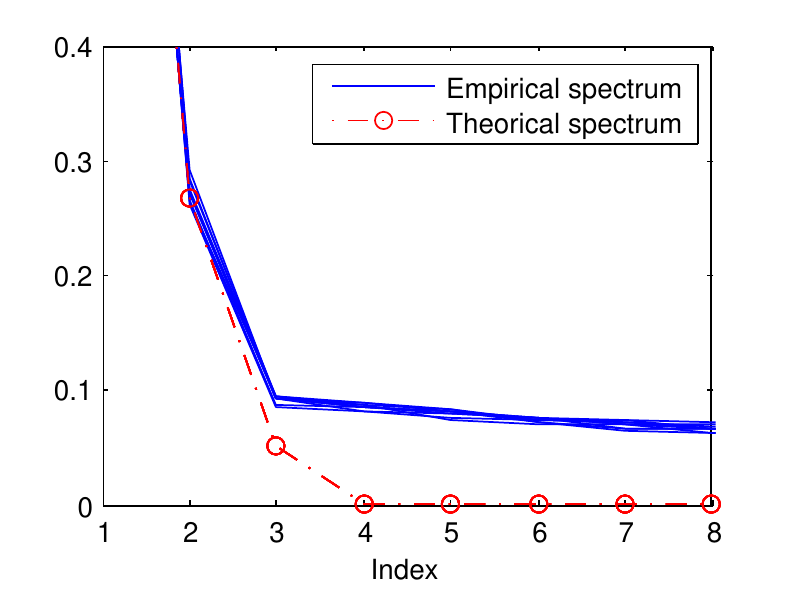}
\caption{$n = 19,998$, Beta parameters $(2; 5)$, $(4; 2)$ and $(4; 4)$.}
\label{fig_spectrum_N_19998A}
\end{subfigure}
\begin{subfigure}[b]{0.49\textwidth}
\centering
\includegraphics[scale=0.8]{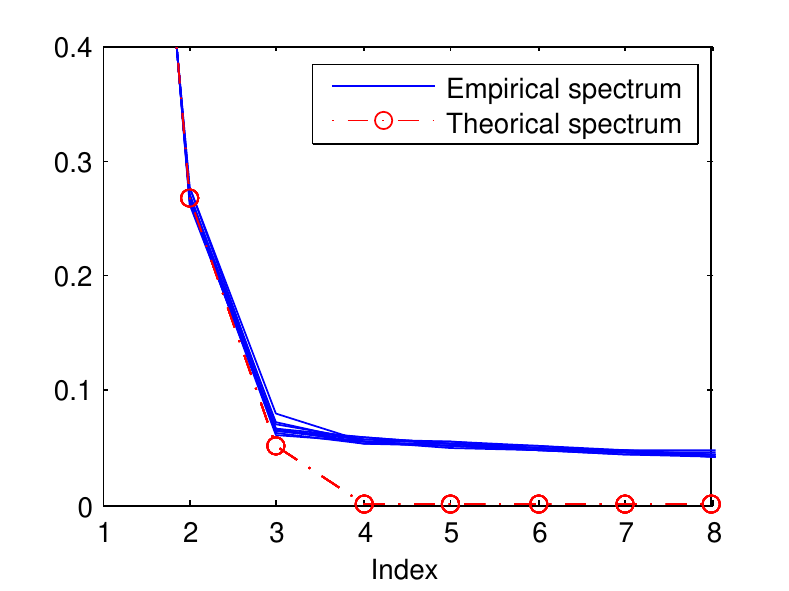}
\caption{$n = 49,998$, Beta parameters $(2; 5)$, $(4; 2)$ and $(4; 4)$.}
\end{subfigure}

\begin{subfigure}[b]{0.49\textwidth}
\centering
\includegraphics[scale=0.8]{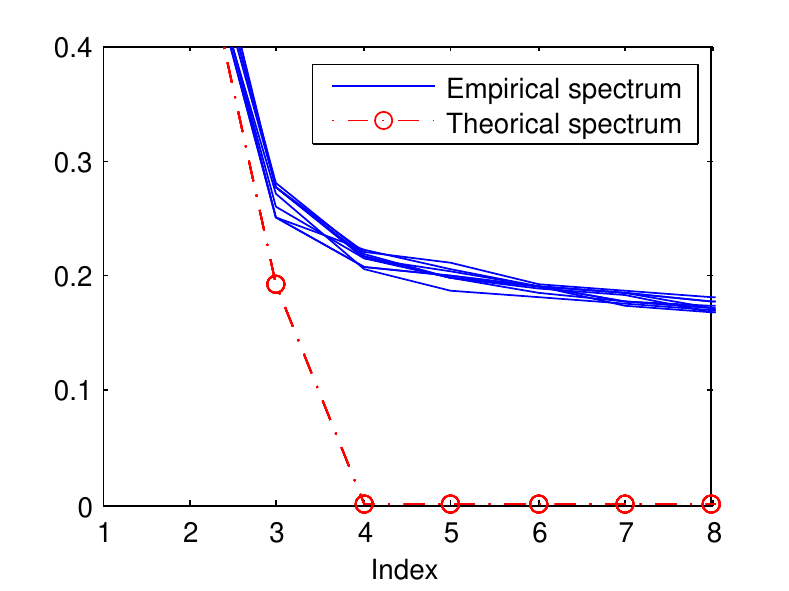}
\caption{$n = 3,000$, Beta parameters $(1.5; 5)$, $(7; 2)$ and $(6; 6)$.}
\end{subfigure}
\begin{subfigure}[b]{0.49\textwidth}
\centering
\includegraphics[scale=0.8]{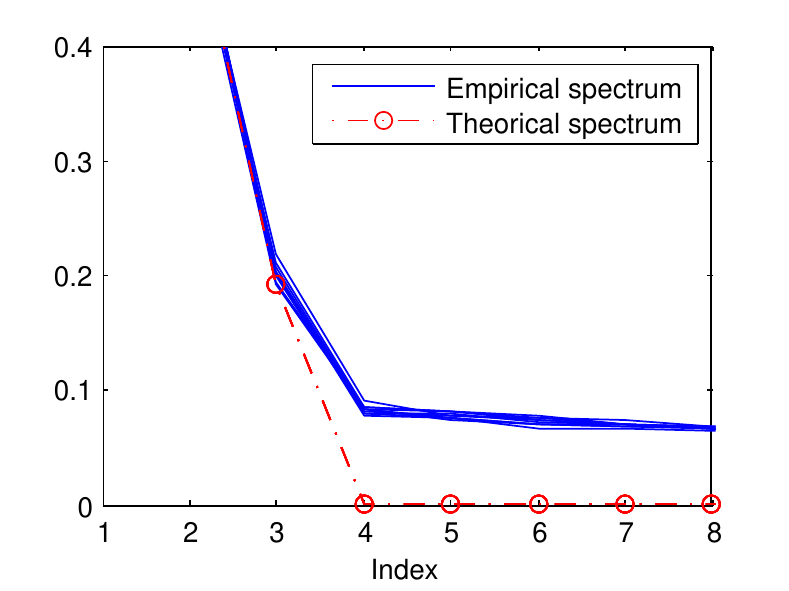}
\caption{$n = 19,998$, Beta parameters $(1.5; 5)$, $(7; 2)$ and $(6; 6)$.}
\label{fig_spectrum_N_19998B}
\end{subfigure}
\caption{Spectrum of the empirical matrix $\hat{\Nbf}_M$ and the theoretical matrix $\Nbf_M$ for $M=40$ and 10 simulations. The first singular values are too large to appear here.}
\label{fig_spectrum_N}
\end{figure}

\begin{figure}[!h]
\centering
\includegraphics[scale=0.8]{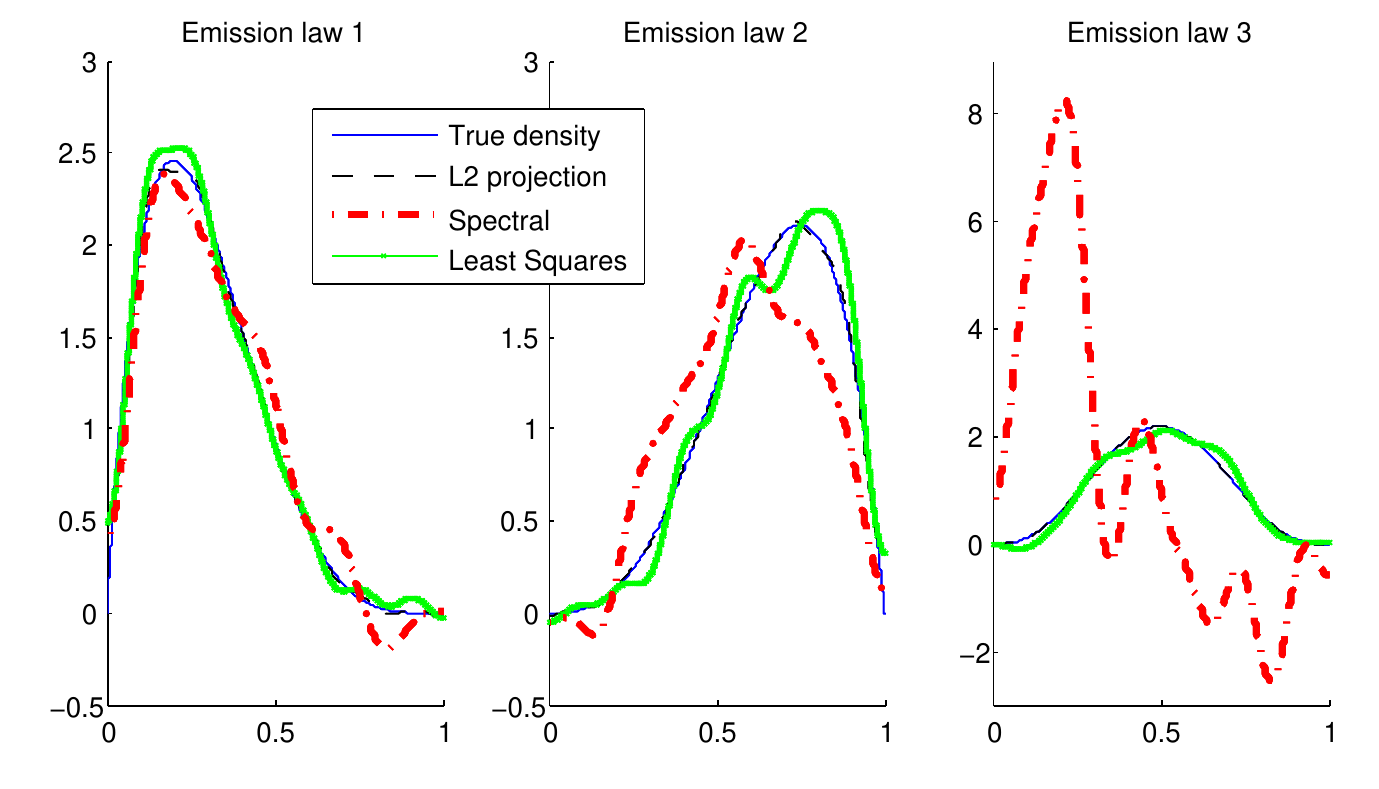}
\caption{Estimators of the emission densities for $n=19,998$ and Beta parameters $(2; 5)$, $(4; 2)$ and $(4; 4)$. We took $K = \hat{K}_\text{l.s.} = 3$ and $M = \hat{M} = 13$. The bad behaviour of the spectral algorithm when the emission densities are poorly separated is clearly visible on the third emission distribution.}
\label{fig_estimators_N19998}
\end{figure}

\begin{figure}[!h]
\centering
\includegraphics[scale=0.8]{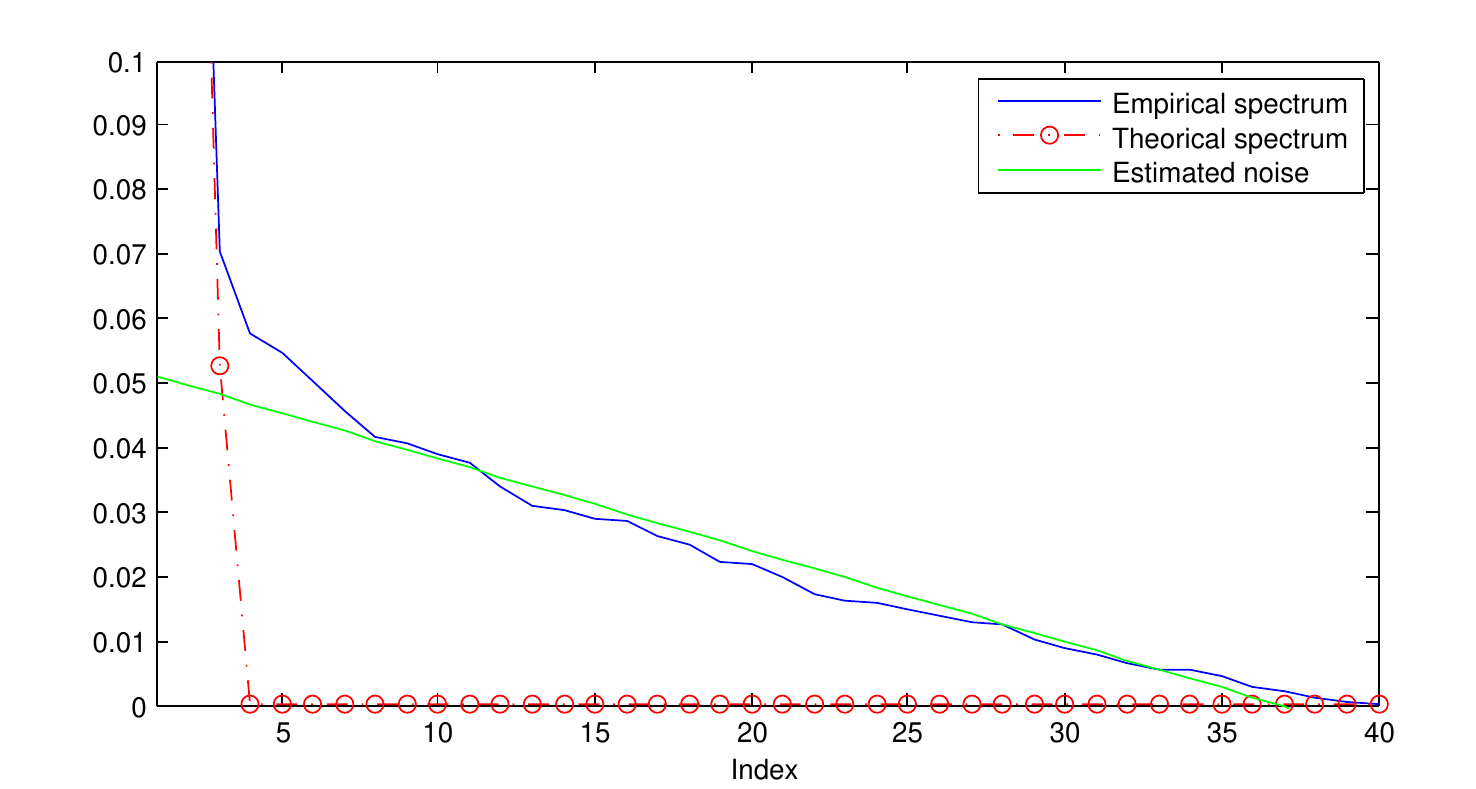}
\caption{Spectrum of $\Nbf_M$ for $M=40$ and $n=49,998$ for Beta parameters $(2; 5)$, $(4; 2)$ and $(4; 4)$. The regression (green line) has been performed on the 35 smallest singular values. The two largest singular values are too large to appear here.}
\label{fig_spectrum_regression}
\end{figure}

Therefore, what we need is a way to threshold the parameters in order to distinguish noise from significant singular values. The estimator $\hat{K}_\text{sp.}(C)$ is one way to achieve this, but the calibration of $C$ is a tricky problem, since the right choice of $C$ depends on the parameters of the HMM. We will use a different method, which relies on the same idea: identifying the noisy singular values which stand out from the others and saying they correspond to nonzero singular values of $\Nbf_M$. Our heuristics relies on the fact that when one sorts the singular values in decreasing order, then the smallest ones approximately follow an affine relation with respect to their index. This tendency is shown in Figure \ref{fig_spectrum_regression}.

We proceed as follows. Let $M$ and $M_\text{reg}$ be two positive integers such that ${M_\text{reg} \leq M \leq M_\text{max}}$. We estimate the affine dependance of the singular values of $\hat{\Nbf}_M$ with respect to their index with a linear regression using its $M_{\text{reg}}$ smallest singular values. Then, we set a thresholding parameter $\tau > 1$. We say a singular value is \emph{significant} if it is above $\tau$ times the value that the regression predicts for it. Lastly, we take $\hat{K}_\text{sp.}$ as the number of consecutive significant singular values starting from the largest one. This heuristics seems to work as soon as $\tau$ is large enough, e.g. $\tau = 1.5$.

\section{Proofs}
\label{sec_proofs}

\subsection{Main technical result}

The following lemma is the main technical result of this paper. It is the key for both the strong consistency and the oracle inequalities. It allows to control the difference between the empirical criterion $\gamma_n$ and the theoretical $\Lbf^2$ loss for all models at the same time.\\

Define $\nu : t \longmapsto \frac{1}{n} \sum_{s=1}^n t(Z_s) - \int t g^*$, so that
\begin{equation}
\label{eq_def_nu}
\forall t \in \Lbf^2(\Ycal^L, \mu^{\otimes L}), \quad
	\gamma_n(t) + \| g^* \|_2^2 = \| t - g^* \|_2^2 - 2 \nu(t)
\end{equation}

Let
\begin{multline}
\label{eq_def_Z}
s = (s_{K,M})_{K, \, M} \in \Sbf := \prod_{K \in \Nbb^*, \, M \in \Mcal} \left(\bigcup_{K} S_K \right) \\
\longmapsto (Z_{K,M}(s))_{K, M} \\
:= \left( \underset{t \in S_{K,M}}{\sup} \left[
	\frac{| \nu (t - s_{K,M}) |}{\| t - s_{K,M} \|_2^2 + x_{K,M}^2}
\right] \right)_{K, M}
\end{multline}

\begin{remark}
It is not necessary to assume that $s_{K,M} \in S_{K,M}$. In particular, one can take $s_{K,M} = g^*$ for all $K, M$. In that case, we will simply write $Z_{K,M}(g^*)$.
\end{remark}

\begin{lemma}
\label{lemma_control_of_Z}
Assume \textbf{[HX]} and \textbf{[HF]} hold. Then there exists a sequence $(x_{K,M})_{K, M} \equiv (x_{K,M})_{K, M}(\cquad, \cinf, \Qbf^*, L)$ and positive constants $(n_0, \rho, A) \equiv (n_0, \rho, A)(\cquad, \cinf, \Qbf^*, L)$ such that if the penalty $\pentilde$ satisfies
\begin{equation*}
\forall n, \; \forall M \leq n, \; \forall K \leq n \quad \pentilde(n,M,K) \geq \rho (MK + K^2 - 1) \frac{\log(n)}{n}
\end{equation*}
then for all $s \in \Sbf$, $n \geq n_0$ and $x > 0$, one has with probability larger than ${1 - e^{-x}}$:
\begin{equation*}
\begin{cases}
\displaystyle \underset{K' \leq n , \, M' \leq n}{\sup} \; Z_{K',M'}(s) \leq \frac{1}{4} \\
\displaystyle \underset{K' \leq n , \, M' \leq n}{\sup} \left(2 Z_{K',M'}(s) x_{K',M'}^2 - \pentilde(n,M',K') \right) \leq A \frac{x}{n}
\end{cases}
\end{equation*}
\end{lemma}

\begin{remark}
One can replace the constant $1/4$ in the first upper bound by any $\epsilon > 0$, at the cost of changing the constants $n_0$, $\rho$ and $A$.
\end{remark}

The structure of the proof follows the usual method to control empirical processes, see for instance \citet{Mas07}, Chapter 6, adapted to the HMM structure by \cite{dCGL15}. The novelty and main difficulty of the proof comes from the generalization to both nonparametric densities and an unknown number of states: we had to introduce a much finer control of the constants and of the bracketing entropy of the models in order to take the dependency in the order of the HMM into account.


The details of the proof can be found in appendix \ref{sec_proof_control_Z}.

\subsection{Identifiability proofs}

\subsubsection{Proof of Corollary \ref{cor_distance_g_ast_S_KM}}
\label{sec_proof_cor_distance}

Denote by $\text{Proj}_A$ the orthogonal projection on a linear space $A$.

Since the union of $(\Pfrak_M)_{M \in \Mcal}$ is dense in $\Fcal$, we can take $M$ such that \textbf{[HidA]} or \textbf{[HidB]} holds for $\fbf^*_M = (f^*_{M,k})_{k \in \Xcal} := (\text{Proj}_{\Pfrak_M} f^*_k)_{k \in \Xcal}$.

We will need the following lemma.

\begin{lemma}
\label{lemma_projection}
\begin{equation*}
\forall \pi \in \Rbb^K, \; \forall \Qbf \in \Rbb^{K \times K}, \; \forall \fbf \in \Fcal^K, \; \forall M, \qquad
	\text{Proj}_{\Pfrak_M^{\otimes L}} \left( g^{\pi, \Qbf, \fbf} \right)
		= g^{\pi, \Qbf, \text{Proj}_{\Pfrak_M}(\fbf)}
\end{equation*}
\end{lemma}

\begin{proof}
By linearity of the projection operator, it is enough to prove that for all $(t_1, \dots, t_L) \in (\Lbf^2(\Ycal, \mu))^L$,
\begin{equation*}
\text{Proj}_{\Pfrak_M^{\otimes L}} (t_1 \otimes \dots \otimes t_L)
	= \text{Proj}_{\Pfrak_M}(t_1) \otimes \dots \otimes \text{Proj}_{\Pfrak_M}(t_L)
\end{equation*}
which is easy to check.
\end{proof}

We will make a proof by contradiction. Assume that $\inf_{t \in S_K} \| t - g^* \|_2 = 0$ for some $K < K^*$. Then there exists a sequence $(g_n)_{n \geq 1} = (g^{\pi_n, \Qbf_n, \fbf_n})_{n \geq 1}$ such that $g_n \longrightarrow g^*$ in $\Lbf^2(\Ycal^L, \mu^{\otimes L})$, with $\pi_n \in \Delta_K$, $\Qbf_n$ a transition matrix of size $K$ and $\fbf_n \in \Fcal^K$.

The orthogonal projection on $\Pfrak_M^{\otimes L}$ is continuous, so by using Lemma \ref{lemma_projection}, one gets that
\begin{equation*}
g^{\pi_n, \Qbf_n, \text{Proj}_{\Pfrak_M}(\fbf_n)}
	\longrightarrow g^{\pi^*, \Qbf^*, \fbf^*_M}
\end{equation*}

Then, using the compacity of $\Delta_K$ and of the set of transition matrices of size $K$ and the relative compacity of $(\Fcal \cap \Pfrak_M)^K$ (which is a bounded subset of a finite dimension linear space), one gets (up to extraction of a subsequence) that there exists $\pi_\infty \in \Delta_K$, $\Qbf_\infty$ a transition matrix of size $K$ and $\fbf_\infty \in (\Pfrak_M)^K$ such that $\pi_n \longrightarrow \pi_\infty$, $\Qbf_n \longrightarrow \Qbf_\infty$ and $\text{Proj}_{\Pfrak_M}(\fbf_n) \longrightarrow \fbf_\infty$.

Finally, using the continuity of the function $(\pi, \Qbf, \fbf) \longmapsto g^{\pi, \Qbf, \fbf}$ and the unicity of the limit, one gets
\begin{equation*}
g^{\pi_\infty, \Qbf_\infty, \fbf_\infty} = g^{\pi^*, \Qbf^*, \fbf^*_M}.
\end{equation*}

Then Proposition \ref{prop_identifiability} contradicts the assumption $K < K^*$, which is enough to conclude.

\subsection{Consistency proofs}

\label{sec_proof_th_underestimation}

The definition of $\hat{K}_\text{l.s.}$ is equivalent to the following one:
\begin{equation*}
\hat{K}_\text{l.s.} \in \underset{K \leq n}{\argmin} \{ \gamma_n(\hat{g}_{K,\hat{M}_K}) + \pen(n,\hat{M}_K,K) \}
\end{equation*}
where
\begin{equation*}
\hat{M}_K \in \underset{M \leq n}{\argmin} \{ \gamma_n(\hat{g}_{K,M}) + \pen(n,M,K) \}
\end{equation*}
Choosing $K$ rather than $K^*$ means that $K$ is better than $K^*$, i.e.
\begin{multline*}
\{ \hat{K}_\text{l.s.} = K \}
	\subset \Big\{ 0 \geq
		\underset{M \leq n}{\inf} \{ \gamma_n(\hat{g}_{K,M}) + \pen(n,M,K) \} \\
		- \underset{M \leq n}{\inf} \{ \gamma_n(\hat{g}_{K^*,M}) + \pen(n,M,K^*) \}
		\Big\}.
\end{multline*}
Let
\begin{align*}
D_{n,K}~:=& \, \underset{M \leq n}{\inf} \{ \gamma_n(\hat{g}_{K,M}) + \pen(n,M,K) \} \\
	& \qquad \qquad - \underset{M \leq n}{\inf} \{ \gamma_n(\hat{g}_{K^*,M}) + \pen(n,M,K^*) \} \\
	=& \, \gamma_n(\hat{g}_{K,\hat{M}_K}) + \pen(n,\hat{M}_K,K) \\
	& \qquad \qquad - \underset{M \leq n}{\inf} \{ \underset{t \in S_{K^*,M}}{\inf} \gamma_n(t) + \pen(n,M,K^*) \}.
\end{align*}
Then
\begin{equation*}
\{ \hat{K}_\text{l.s.} = K \} \subset \{ D_{n,K} \leq 0 \}.
\end{equation*}
We will thus control the probability of the latter event for all $K < K^*$ in the first case and $K > K^*$ in the second case.

\paragraph{Proof of Theorem \ref{th_underestimation}}

Let $M_0 \in \Mcal$. We will choose a suitable value for this integer later in the proof. Assume $n \geq M_0$. Then by definition of $D_{n,K}$ and of $\nu$ (equation (\ref{eq_def_nu})),
\begin{align*}
D_{n,K}
	&\geq \gamma_n(\hat{g}_{K,\hat{M}_K}) + \pen(n,\hat{M}_K,K) - \gamma_n(g^*_{K^*,M_0}) - \pen(n,M_0,K^*)) \\
	&\geq \| g^* - \hat{g}_{K,\hat{M}_K} \|_2^2
			- \| g^* - g^*_{K^*,M_0} \|_2^2
			- 2 \nu(\hat{g}_{K,\hat{M}_K} - g^*_{K^*,M_0}) \\
	&\qquad \qquad \qquad + \pen(n,\hat{M}_K,K) - \pen(n,M_0,K^*).
\end{align*}
Using the definition of $Z_{K,M}$ (equation (\ref{eq_def_Z})), one gets that
\begin{align*}
| \nu(\hat{g}_{K,\hat{M}_K} - g^*_{K^*,M_0}) |
	\leq\; & | \nu(\hat{g}_{K,\hat{M}_K} - g^*) | + | \nu(g^* - g^*_{K^*,M_0}) | \\
	\leq\; & Z_{K,\hat{M}_K}(g^*) ( \| g^* - \hat{g}_{K,\hat{M}_K} \|_2^2 + x_{K,\hat{M}_K}^2 ) \\
		&+ Z_{K^*, M_0}(g^*) ( \| g^* - g^*_{K^*,M_0} \|_2^2 + x_{K^*, M_0}^2 ).
\end{align*}
Let $n_0$, $\rho$ and $A$ be as in Lemma \ref{lemma_control_of_Z}. We can assume that $n_0 \geq K^*$ so that $K^* \leq n$. Let us introduce the function $\pentilde(n,M,K) = \rho (MK + K^2 - 1) \frac{\log(n)}{n}$. Let $n \geq n_0$ and $x>0$ and assume we are in the event of probability $1 - e^{-x}$ of Lemma \ref{lemma_control_of_Z}. Then, for all $K \leq n$:
\begin{align*}
| \nu(\hat{g}_{K,\hat{M}_K} - g^*_{K^*,M_0}) |
	\; \leq \; & \frac{1}{4}  \| g^* - \hat{g}_{K,\hat{M}_K} \|_2^2
			+ \frac{1}{2} A \frac{x}{n}
			+ \frac{1}{2} \pentilde(n, \hat{M}_K, K) \\
		&+ \frac{1}{4} \| g^* - g^*_{K^*,M_0} \|_2^2
			+ \frac{1}{2} A \frac{x}{n}
			+ \frac{1}{2} \pentilde(n, M_0, K^*)
\end{align*}
and
\begin{multline*}
D_{n,K} \geq \frac{1}{2} \| g^* - \hat{g}_{K,\hat{M}_K} \|_2^2
			- \frac{3}{2} \| g^* - g^*_{K^*,M_0} \|_2^2
			- 2 A \frac{x}{n}
			+ \pen(n,\hat{M}_K,K) \\
			- \pen(n,M_0,K^*)
			- \pentilde(n, \hat{M}_K, K)
			- \pentilde(n, M_0, K^*).
\end{multline*}
We assumed $\pen \geq \pentilde$, so that
\begin{align*}
D_{n,K} &\geq \frac{1}{2} \| g^* - \hat{g}_{K,\hat{M}_K} \|_2^2
			- \frac{3}{2} \| g^* - g^*_{K^*,M_0} \|_2^2
			- 2 A \frac{x}{n}
			- 2 \pen(n,M_0,K^*)
\end{align*}
Corollary \ref{cor_distance_g_ast_S_KM} ensures that
\begin{equation*}
d := \inf_{K < K^*} \inf_{t \in S_K} \| t - g^* \|_2 > 0,
\end{equation*}
so that for all $K < K^*$,
\begin{align*}
D_{n,K} &\geq \frac{d^2}{2}
			- \frac{3}{2} \| g^* - g^*_{K^*,M_0} \|_2^2
			- 2 A \frac{x}{n}
			- 2 \pen(n,M_0,K^*).
\end{align*}
By density of $(\Pfrak_M)_{M \in \Mcal}$ in $\Fcal$, one gets that
\begin{equation*}
\inf_M \| g^*_{K^*,M} - g^* \|_2 = 0
\end{equation*}
so that there exists $M_0$ such that $\| g^* - g^*_{K^*,M_0} \|_2^2 \leq d^2 / 6$. If we choose this $M_0$, we get that
\begin{align*}
D_{n,K} &\geq \frac{d^2}{4}
			- 2 A \frac{x}{n}
			- 2 \pen(n,M_0,K^*).
\end{align*}
Which implies that $D_{n,K} > 0$ as soon as $2 A x / n < d^2 / 4 - 2 \pen(n, M_0, K^*)$, i.e.
\begin{equation*}
x < \left(\frac{d^2}{8} - \pen(n,M_0,K^*) \right) \frac{n}{A}.
\end{equation*}
To conclude, note that there exists $\tilde{n_0} \geq \max(n_0, M_0)$ such that for all $n \geq \tilde{n_0}$, $\pen(n,M_0,K^*) \leq \frac{d^2}{16}$. Then, letting $\beta = \frac{d^2}{16 A}$, one has for all $n \geq \tilde{n_0}$, with probability $1 - e^{-\beta n}$, for all $K < K^*$, $D_{n,K} > 0$, which implies that $\hat{K}_\text{l.s.} \neq K$.

\paragraph{Proof of Theorem \ref{th_overestimation}}

For all $K \geq K^*$,
\begin{align*}
D_{n,K} \geq \gamma_n(\hat{g}_{K,\hat{M}_K}) + \pen(n,\hat{M}_K,K)
	- \gamma_n(g^*_{K^*,\hat{M}_K}) - \pen(n,\hat{M}_K,K^*)
\end{align*}
and
\begin{multline*}
\gamma_n(\hat{g}_{K,\hat{M}_K}) - \gamma_n(g^*_{K^*,\hat{M}_K})
	= \| \hat{g}_{K,\hat{M}_K} - g^* \|_2^2 - \| g^*_{K^*,\hat{M}_K} - g^* \|_2^2 \\
	- 2 \nu(\hat{g}_{K,\hat{M}_K} - g^*_{K^*,\hat{M}_K}).
\end{multline*}
Note that $g^*_{K^*,\hat{M}_K} = g^*_{K,\hat{M}_K}$ is the orthogonal projection of $g^*$ on $\Pfrak_{\hat{M}_K}^{\otimes L}$ and $\hat{g}_{K,\hat{M}_K} \in S_{K, \hat{M}_K} \subset \Pfrak_{\hat{M}_K}^{\otimes L}$, so that, using the Pythagorean Theorem,
\begin{equation*}
\| \hat{g}_{K,\hat{M}_K} - g^* \|_2^2 - \| g^*_{K^*,\hat{M}_K} - g^* \|_2^2 = \| \hat{g}_{K,\hat{M}_K} - g^*_{K^*,\hat{M}_K} \|_2^2.
\end{equation*}
Let $n_0$, $\rho$ and $A$ be as in Lemma \ref{lemma_control_of_Z}. We can assume that $n_0 \geq K^*$ so that $K^* \leq n$. Let us introduce the function $\pentilde(n,M,K) = \rho (MK + K^2 - 1) \frac{\log(n)}{n}$. Let $n \geq n_0$ and $x>0$ and assume we are in the event of probability $1 - e^{-x}$ of Lemma \ref{lemma_control_of_Z}. Then, for all $K \leq n$ such that $K \geq K^*$:
\begin{align*}
| \nu(\hat{g}_{K,\hat{M}_K} - g^*_{K^*,\hat{M}_K}) |
	&= | \nu(\hat{g}_{K,\hat{M}_K} - g^*_{K,\hat{M}_K}) | \\
	&\leq Z_{K,\hat{M}_K}((g^*_{K',M'})_{K', M'}) \| \hat{g}_{K,\hat{M}_K} - g^*_{K,\hat{M}_K} \|_2^2 \\
	&\qquad \qquad + Z_{K,\hat{M}_K}((g^*_{K',M'})_{K', M'}) x_{K,\hat{M}_K}^2 \\
	&\leq \frac{1}{4} \| \hat{g}_{K,\hat{M}_K} - g^*_{K,\hat{M}_K} \|_2^2 + \frac{1}{2} A \frac{x}{n} + \frac{1}{2} \pentilde(n,\hat{M}_K,K),
\end{align*}
which implies
\begin{align*}
\gamma_n(\hat{g}_{K,\hat{M}_K}) - \gamma_n(g^*_{K^*,\hat{M}_K})
	&\geq \frac{1}{2} \| \hat{g}_{K,\hat{M}_K} - g^*_{K,\hat{M}_K} \|_2^2 - A \frac{x}{n} - \pentilde(n,\hat{M}_K,K) \\
	&\geq - A \frac{x}{n} - \pentilde(n,\hat{M}_K,K)
\end{align*}
so that for all $K \leq n$ such that $K \geq K^*$:
\begin{align*}
D_{n,K} \geq \pen(n, \hat{M}_K, K) - \pen(n, \hat{M}_K, K^*) - \pentilde (n, \hat{M}_K, K) - A \frac{x}{n}.
\end{align*}
Now, assume that \textbf{[Hpen]}($\alpha, \rho$) holds for some $\alpha > 0$ and the above constant $\rho$. Then there exists $n_1$ such that for all $n \geq n_1$ and for all $K \leq n$ such that $K \geq K^*$,
\begin{align*}
D_{n,K} \geq \alpha \frac{\log(n)}{n} - A \frac{x}{n},
\end{align*}
which is strictly positive as soon as $x < \alpha \log(n) / A$.
Thus, letting $\beta = 1/(2A)$, one has for all $n \geq \max(n_0, n_1, K^*)$, with probability $1 - n^{- \beta \alpha}$, for all $K \leq n$ such that $K > K^*$, $D_{n,K} > 0$, which implies that $\hat{K}_\text{l.s.} \neq K$. This concludes the proof.

\section{Acknowledgment}

We would like to thank Elisabeth Gassiat for her precious advice and Yohann de Castro for his codes which were at the root of our numerical experiments.


\bibliographystyle{imsart-nameyear}
\bibliography{these}

\newpage
\appendix

\section{Spectral algorithm}
\label{app_algo}

{
\begin{algorithm}[!h]
  \SetAlgoLined
  \KwData{An observed chain $(Y_{1},\ldots,Y_{n})$ and an order $K$.}
  \KwResult{Spectral estimators $\hat\pi$, $\hat\Qbf$ and the estimators $(\hat f_{M,k})_{k\in\Xcal}$ of $(f^*_k)_{k\in\Xcal}$ in $\Pfrak_M$ (equipped with an orthonormal basis $\Phi_M = (\varphi_1, \dots, \varphi_M)$).} 
    \BlankLine
\begin{enumerate}
\item[{\bf [Step 1]}] Consider the following empirical estimators: for any $a,b,c$ in $\{1,\ldots,M\}$,
	\begin{itemize}
	\item $\hat{\Lbf}_M(a):=\frac{1}{n} \sum_{s=1}^{n}\varphi_{a}(Y_{1}^{(s)})$,
	\item $\hat{\Mbf}_M(a,b,c):=\frac{1}{n} \sum_{s=1}^{n}\varphi_{a}(Y_{1}^{(s)})\varphi_{b}(Y_{2}^{(s)})\varphi_{c}(Y_{3}^{(s)})$,
	\item $\hat{\Nbf}_M(a,b):=\frac{1}{n} \sum_{s=1}^{n}\varphi_{a}(Y_{1}^{(s)})\varphi_{b}( Y_{2}^{(s)})$,
	\item $\hat{\Pbf}_M(a,c):=\frac{1}{n} \sum_{s=1}^{n}\varphi_{a}(Y_{1}^{(s)})\varphi_{c}(Y_{3}^{(s)})$.
	\end{itemize}
\item[{\bf [Step 2]}]
Let $\hat{\Ubf}$ be the $M\times K$ matrix of orthonormal right singular vectors of $\hat{\Pbf}_M$ corresponding to its top $K$ singular values. 
\item[{\bf [Step 3]}]
Form the matrices $\hat{\Bbf}(b):=(\hat{\Ubf}^{\top}\hat{\Pbf}_M\hat{\Ubf})^{-1}\hat{\Ubf}^{\top}\hat{\Mbf}_M(\ldotp,b,\ldotp)\hat{\Ubf}$ for all $b\in\{1,\ldots,M\}$.
\item[{\bf [Step 4]}]
Set $\Theta$ a $(K\times K)$ uniformly drawn random unitary matrix and form the matrices ${\hat{\Cbf}}(k):=\sum_{b=1}^{M}(\hat{\Ubf}\Theta)(b,k)\hat{\Bbf}(b)$ for all $k\in\{1,\ldots,K\}$.
\item[{\bf [Step 5]}]
Compute $\hat{\Rbf}$ a $(K\times K)$ unit Euclidean norm columns matrix that diagonalizes the matrix $\hat{\Cbf}(1)$: $\hat{\Rbf}^{-1}\hat{\Cbf}(1)\hat{\Rbf}=\text{Diag}{(\hat\Lambda(1,1),\ldots,\hat\Lambda(1,K))}$.
\item[{\bf [Step 6]}]
Set $\hat\Lambda(k,k'):=(\hat{\Rbf}^{-1}\hat{\Cbf}(k)\hat{\Rbf})(k',k')$ for all $k,k'\in\Xcal$ and $\hat{\Obf}_M:=\hat{\Ubf}\Theta\hat\Lambda$.
\item[{\bf [Step 7]}]
Consider the emission distributions estimator $\hat\fbf:=(\hat f_{M,k})_{k\in\Xcal}$ defined by $\hat f_{M,k}:=\sum_{m=1}^{M}\hat{\Obf}_M(m,k)\varphi_{{m}}$ for all $k\in\Xcal$.
\item[{\bf [Step 8]}]
Set $\tilde\pi := \Pi_{\Delta_K}\Big(\big(\hat{\Ubf}^{\top}\hat{\Obf}_M\big)^{-1}\hat{\Ubf}^{\top}\hat{\Lbf}_M\Big)$ where $\Pi_{\Delta_K}$ denotes the projection onto the simplex in dimension $K$.
\item[{\bf [Step 9]}]
Consider the transition matrix estimator: 
\[
\hat\Qbf:=\Pi_{\mathrm{TM}}\Big(\big(\hat{\Ubf}^{\top}\hat{\Obf}_M\text{Diag}{\tilde\pi}\big)^{-1}\hat{\Ubf}^{\top}\hat{\Nbf}_M\hat{\Ubf}\big(\hat{\Obf}_M^{\top}\hat{\Ubf}\big)^{-1}\Big)\,,
\] where $\Pi_{\mathrm{TM}}$ denotes the projection 
onto the convex set of transition matrices, and define
$\hat\pi$ as the stationary distribution of $\hat\Qbf$.
\end{enumerate}
  \caption{Spectral estimation of HMM parameters (\cite{dCGL15}, \cite{dCGLLC15})}
\label{alg_Spectral}
\end{algorithm}}

\section{Proofs of the oracle inequalities}
\label{app_oracle_proof}

\subsection{Proof of Theorem \ref{th_oracle_g}}
\label{sec_proof_th_oracle}

Let $K \leq n$ and $M \leq n$. Then
\begin{align*}
\gamma_n(\hat{g}) + \pen(n,\hat{M},\hat{K}_\text{l.s.})
	&\leq \gamma_n(\hat{g}_{K,M}) + \pen(n,M,K) \\
	&\leq \gamma_n(g^*_{K,M}) + \pen(n,M,K)
\end{align*}
where the first inequality comes from the definition of $(\hat{K}_\text{l.s.}, \hat{M})$ and the second from the definition of $\hat{g}_{K,M}$. Therefore,
\begin{align*}
\gamma_n(\hat{g}) - \gamma_n(g^*_{K,M})
&\leq \pen(n,M,K) - \pen(n,\hat{M},\hat{K}_\text{l.s.}).
\end{align*}
By definition of $\nu$ (equation \ref{eq_def_nu}),
\begin{align*}
\gamma_n(t_1) - \gamma_n(t_2)
	= \| t_1 - g^* \|_2^2 - \| t_2 - g^* \|_2^2
		- 2 \nu(t_1 - t_2)
\end{align*}
so that
\begin{multline*}
\| \hat{g} - g^* \|_2^2 
\leq \| g^*_{K,M} - g^* \|_2^2
	+ \pen(n,M,K) - \pen(n,\hat{M},\hat{K}_\text{l.s.}) \\
	+ 2 \nu(\hat{g}_{\hat{M},\hat{K}_\text{l.s.}} - g^*_{K,M})
\end{multline*}
Now we want to control the $\nu$ term. By linearity,
\begin{align*}
\nu(\hat{g}_{\hat{K}_\text{l.s.},\hat{M}} - g^*_{K,M})
	&= \nu(\hat{g}_{\hat{K}_\text{l.s.},\hat{M}} - g^*)
		+ \nu(g^* - g^*_{K,M})
\end{align*}
Using the definition of $Z_{K,M}$ (equation \ref{eq_def_Z}), we get that 
\begin{equation*}
\begin{cases}
| \nu(\hat{g}_{\hat{K}_\text{l.s.},\hat{M}} - g^*) |
	\leq Z_{\hat{K}_\text{l.s.},\hat{M}}(g^*) ( \| \hat{g}_{\hat{K}_\text{l.s.},\hat{M}} - g^* \|_2^2 + x_{\hat{K}_\text{l.s.},\hat{M}}^2) \\
| \nu(g^*_{K,M} - g^*) |
	\leq Z_{K,M}(g^*) ( \| g^*_{K,M} - g^* \|_2^2 + x_{K,M}^2)
\end{cases}
\end{equation*}
so that, using Lemma \ref{lemma_control_of_Z}, for all $n \geq n_0$ and $x>0$, with probability larger than $1 - e^{-x}$, for all $M \leq n$ and $K \leq n$,
\begin{multline*}
| \nu(\hat{g}_{\hat{K}_\text{l.s.},\hat{M}} - g^*_{K,M}) |
	\leq \frac{1}{4} \| \hat{g} - g^* \|_2^2 + \frac{1}{4} \| g^*_{K,M} - g^* \|_2^2 + A \frac{x}{n} \\ + \frac{1}{2} \pen(n,\hat{M},\hat{K}_\text{l.s.}) + \frac{1}{2} \pen(n,M,K)
\end{multline*}
so that
\begin{align*}
\| \hat{g} - g^* \|_2^2 
\leq& \| g^*_{K,M} - g^* \|_2^2 + 2 \pen(n,M,K) \\
	& + \frac{1}{2} \| \hat{g} - g^* \|_2^2
	+ \frac{1}{2} \| g^*_{K,M} - g^* \|_2^2 + 2 A \frac{x}{n},
\end{align*}
which means that
\begin{align*}
\frac{1}{2} \| \hat{g} - g^* \|_2^2 
\leq& \frac{3}{2} \| g^*_{K,M} - g^* \|_2^2 + 2 \pen(n,M,K) + 2 A \frac{x}{n}
\end{align*}
and finally
\begin{align*}
\| \hat{g} - g^* \|_2^2 
\leq& 4 \underset{K \leq n, \, M \leq n}{\inf} \{ \| g^*_{K,M} - g^* \|_2^2 + \pen(n,M,K) \} + 4 A \frac{x}{n}
\end{align*}
which is the expected inequality.

\subsection{Proof of equation \ref{eq_upperbound}}
\label{sec_proof_eq_upperbound}

First, decompose the difference in three terms.
\begin{multline*}
\| g^{\pi_1, \Qbf_1, \fbf_1} - g^{\pi_2, \Qbf_2, \fbf_2} \|_2
	\leq \| g^{\pi_1, \Qbf_1, \fbf_1} - g^{\pi_2, \Qbf_1, \fbf_1} \|_2
		+ \| g^{\pi_2, \Qbf_1, \fbf_1} - g^{\pi_2, \Qbf_2, \fbf_1} \|_2 \\
		+ \| g^{\pi_2, \Qbf_2, \fbf_1} - g^{\pi_2, \Qbf_2, \fbf_2} \|_2
\end{multline*}
Then we can control each term separately. Let $(\varphi_m)_{m \in \Nbb^*}$ be an orthonormal basis of $\cup_M \Pfrak_M$.
\begin{align*}
\| g^{\pi_1, \Qbf, \fbf} - g^{\pi_2, \Qbf, \fbf} \|_2^2
	=& \left\| \sum_{\kbf \in \Xcal^L} (\pi_1 - \pi_2)_{k_1}
			\Qbf_{k_1, k_2} \dots \Qbf_{k_{L-1}, k_L}
			\bigotimes_{i=1}^L f_{k_i} \right\|_2^2 \\
	=& \sum_{\mbf \in (\Nbb^*)^L} \left( \sum_{\kbf \in \Xcal^L} (\pi_1 - \pi_2)_{k_1}
			\Qbf_{k_1, k_2} \dots \Qbf_{k_{L-1}, k_L}
			\prod_{i=1}^L \langle f_{k_i}, \varphi_{m_i} \rangle \right)^2 \\
	\leq& \sum_{\kbf \in \Xcal^L} (\pi_1 - \pi_2)_{k_1}^2
			\Qbf_{k_1, k_2} \dots \Qbf_{k_{L-1}, k_L} \\
		&\times \sum_{\kbf' \in \Xcal^L} \Qbf_{k'_1, k'_2} \dots \Qbf_{k'_{L-1}, k'_L}
			\prod_{i=1}^L \sum_{m_i \in \Nbb^*} \langle f_{k'_i}, \varphi_{m_i} \rangle^2
\end{align*}
using Cauchy-Schwarz inequality. Then, since $\sum_{m_i \in \Nbb^*} \langle f_{k'_i}, \varphi_{m_i} \rangle^2 = \|f_{k'_i}\|_2^2 \leq \cquad^2$ by \textbf{[HF]} and $\Qbf$ is a transition matrix, we get that
\begin{align*}
\| g^{\pi_1, \Qbf, \fbf} - g^{\pi_2, \Qbf, \fbf} \|_2^2
	\leq K \cquad^{2L} \| \pi_1 - \pi_2 \|^2
\end{align*}
A similar decomposition leads to
\begin{align*}
\| g^{\pi, \Qbf_1, \fbf} - g^{\pi, \Qbf_2, \fbf} \|_2^2
	\leq (L-1) K \cquad^{2L} \| \Qbf_1 - \Qbf_2 \|_F^2
\end{align*}
and
\begin{align*}
\| g^{\pi, \Qbf, \fbf_1} - g^{\pi, \Qbf, \fbf_2} \|_2^2
	\leq L K \cquad^{2(L-1)} \sum_{k \in \Xcal} \| (\fbf_1)_k - (\fbf_2)_k \|_2^2
\end{align*}
These inequalities remain true if the states of the second set of parameters are swapped. Then, we use that $\cquad \geq 1$ by \textbf{[HF]} to conclude.

\subsection{Proof of Lemma \ref{lemma_H_nonzeropolynomial}}
\label{sec_proof_H_nonzeropolynomial}

In the following, we will identify the quadratic form $M$ derived from the second order expansion of $x \mapsto \Nfrak(x)$ and its matrix. Likewise, we will identify the quadratic form $M_\Cfrak$ derived from the second order expansion of $x \mapsto \Nfrak(I_\Cfrak(x))$ with its matrix. Without loss of generality, one can assume $L=3$.

\paragraph{Choice of parameters and expression of $M$.}

Let $\pi \in \Delta_{K^*}$ be the uniform distribution on $\Xcal$, $\Qbf = \Id_{K^*}$ and $\fbf$ such that $\langle f_i, f_j \rangle = F \one_{i=j}$ for some constant $F > 0$. For instance, the $f_i$'s are $F$ times the indicating functions of distinct measurable sets with same measure $\frac{1}{F}$ for $\mu$. In that case, $(f_i / \sqrt{F})_i$ is an orthonormal basis, and the quantity $g^{\pi + p, \Qbf + q, \fbf + A \fbf} - g^{\pi, \Qbf, \fbf}$ can be broken down into three order one terms in $p$, $q$ and $A$:
\begin{itemize}
\item the term in $p$: $\sum_i p_i f_i \otimes f_i \otimes f_i$ ;
\item the term in $q$: $\sum_{i,k} q(i,k) f_i \otimes (f_i + f_k) \otimes f_k$ ;
\item the term in $A$: $\sum_i \left( (Af)_i \otimes f_i \otimes f_i + f_i \otimes (Af)_i \otimes f_i + f_i \otimes f_i \otimes (Af)_i \right)$.
\end{itemize}

Now we can make the list of all second-order terms in the expansion of the quantity ${\| g^{\pi + p, \Qbf + q, \fbf + A \fbf} - g^{\pi, \Qbf, \fbf} \|_2^2}$:
\begin{itemize}
\item $p$ and $p$: $F^3 \sum_i p_i^2$ ;
\item $p$ and $q$: $2 F^3 \sum_i p_i q(i,i)$ ;
\item $p$ and $A$: $3 F^3 \sum_i p_i A_{i,i}$ ;
\item $q$ and $q$: $2 F^3 \sum_{i,k} q(i,k)^2 + 2 F^3 \sum_i q(i,i)^2 $ ;
\item $q$ and $A$: $F^3 \sum_{i,k} q(i,k) A_{k,i} 
		+ F^3 \sum_{i,k} q(i,k) A_{i,k}
		+ 4 F^3 \sum_{i} q(i,i) A_{i,i}$ ;
\item $A$ and $A$: $6 F^3 \sum_i A_{i,i}^2 + 3 F^3 \sum_{i,k} A_{i,k}^2$.
\end{itemize}
We can now write the matrix $M$. In order to clarify the structure of this matrix, let us swap the components of the parameters $(p,q,A)$ and consider the new parameters $(A_\text{diag}, A_\text{else}, p, q_\text{diag}, q_\text{else})$, where $A_\text{diag}$ (resp. $q_\text{diag}$) is a vector of size $K^*$ containing the diagonal coefficients of $A$ (resp. $q$) and $A_\text{else}$ (resp. $q_\text{else}$) contains its other coefficients. Then the matrix is:
\begin{equation*}
M_\text{swapped} = F^3 \left(
\begin{array}{c c | c | c c}
9 \Id_{K^*} & 0  & 3\Id_{K^*} & 6\Id_{K^*} & 0 \\
0 & 3\Id_{{K^*}({K^*}-1)} & 0 & 0 & X \\
\hline
3\Id_{K^*} & 0 & \Id_{K^*} & 2\Id_{K^*} & 0 \\
\hline
6\Id_{K^*} & 0 & 2\Id_{K^*} & 4\Id_{K^*} & 0 \\
0 & X & 0 & 0 & 2\Id_{{K^*}({K^*}-1)}
\end{array}
\right),
\end{equation*}
where $X [(A_{i,j})_{i \neq j}] = (A_{i,j} + A_{j,i})_{i \neq j}$.

\paragraph{Kernel of $M$.}

Substracting the first block of lines to the third and fourth blocks of lines and then the first block of columns to the third and fourth blocks of columns does not change the rank and leads to the matrix
\begin{equation*}
F^3 \left(
\begin{array}{c c | c | c c}
9 \Id_K & 0  & 0 & 0 & 0 \\
0 & 3\Id_{K(K-1)} & 0 & 0 & X \\
\hline
0 & 0 & 0 & 0 & 0 \\
\hline
0 & 0 & 0 & 0 & 0 \\
0 & X & 0 & 0 & 2\Id_{K(K-1)}
\end{array}
\right)
\end{equation*}
Thus $\dim(\Ker (M)) \geq 2K$, where $\Ker(M)$ is the kernel of $M$ and $\dim$ denotes the dimension. If one takes away the lines and columns corresponding to $p$ and $q_\text{diag}$, one gets the matrix
\begin{equation*}
F^3 \left(
\begin{array}{c c c}
9 \Id_{K^*} & 0 & 0 \\
0 & 3\Id_{{K^*}({K^*}-1)} & X \\
0 & X & 2\Id_{{K^*}({K^*}-1)}
\end{array}
\right).
\end{equation*}
This matrix is invertible. Therefore, $\dim(\Ker (M)) = 2K$. Now, for all $i \in [K^*]$, let $e_i^1$ and $e_i^2$ be the vectors defined as
\begin{equation*}
\begin{cases}
(e_i^1)_{p_k} = 0 \qquad \text{for all $k$} \\
(e_i^1)_{A_{k,l}} = 0 \qquad \text{for all $(k,l) \neq (i,i)$} \\
(e_i^1)_{q(k,l)} = 0 \qquad \text{for all $(k,l) \neq (i,i)$} \\
(e_i^1)_{A_{i,i}} = 2 \\
(e_i^1)_{q(i,i)} = -3
\end{cases}
\end{equation*}
and
\begin{equation*}
\begin{cases}
(e_i^2)_{p_k} = 0 \qquad \text{for all $k \neq i$} \\
(e_i^2)_{A_{k,l}} = 0 \qquad \text{for all $(k,l) \neq (i,i)$} \\
(e_i^2)_{q(k,l)} = 0 \qquad \text{for all $(k,l)$} \\
(e_i^2)_{A_{i,i}} = 1 \\
(e_i^2)_{p_i} = -3
\end{cases}
\end{equation*}
One can easily check that these vectors are linearly independant and are all in $\Ker(M)$. Thus, they are a basis of the kernel of $M$: $\Ker(M) = \Span(\{e_i^1, e_i^2 \, | \, i \in [K]\})$.

\paragraph{Nondegeneracy of $M$ restricted on $\Cfrak$.}

Since $M$ is symmetric, and thus diagonalisable in an orthonormal basis,
\begin{equation}
M = P_{\Ker(M)^\bot}^\top M_{\Ker(M)^\bot} P_{\Ker(M)^\bot}
\end{equation}
where $P_{\Ker(M)^\bot}$ is the orthogonal projection on the space of vectors orthogonal to $\Ker(M)$ and $M_{\Ker(M)^\bot}$ is a symmetric positive definite matrix, whose smallest eigenvalue will be written $c_0$ in the following. The last step to conclude will require the two following lemmas:

\begin{lemma}
\label{lemma_C_supplementaire_KerM}
$\Ker(M) \cap \Cfrak = \{0\}$.
\end{lemma}
\begin{proof}
Let $x \in \Ker(M) \cap \Cfrak$, then $x = \sum_i (\lambda_i e_i^1 + \mu_i e_i^2)$ because $(e_i^1, e_i^2)_i$ is a basis of $\Ker(M)$. Since $x \in \Cfrak$, one gets $\lambda_i = 0$ for all $i$ because of the conditions on $q$. Then, the conditions on $A$ imply $\mu_i = 0$ for all $i$, so that $x = 0$.
\end{proof}

\begin{lemma}
\label{lemma_minoration_vsing_restrictions}
There exists a constant $\kappa > 0$ such that for all $x \in \Cfrak$,
\begin{equation}
\| P_{\Ker(M)^\bot} x \|_F^2 \geq \kappa \| x \|_F^2.
\end{equation}
\end{lemma}
\begin{proof}
$P_{\Ker(M)^\bot}$ is continuous. By compacity, the quantity
\begin{equation*}
\kappa := \inf \{ \| P_{\Ker(M)^\bot} x \|_F^2 \, | \, x \in \Cfrak, \|x\|_F^2 = 1\}
\end{equation*}
is reached for some $x_0 \in \Cfrak \setminus \{0\}$. If $\kappa = 0$, then $x_0 \in \Ker(M)$, but this is impossible because of Lemma \ref{lemma_C_supplementaire_KerM}. Therefore $\kappa > 0$.
\end{proof}

Finally, for all $x \in \Cfrak$,
\begin{align*}
x^\top M x
	=& x^\top P_{\Ker(M)^\bot}^\top M_{\Ker(M)^\bot} P_{\Ker(M)^\bot} x \\
	=& (P_{\Ker(M)^\bot} x)^\top M_{\Ker(M)^\bot} (P_{\Ker(M)^\bot} x) \\
	\geq& c_0 \| P_{\Ker(M)^\bot} x \|_F^2 \\
	\geq& c_0 \kappa \| x \|_F^2.
\end{align*}
Therefore, the quadratic form with matrix $M$ is nondegenerate on $\Cfrak$, which shows that $H$ is non-zero for these $(\pi, \Qbf, \fbf)$. To conclude, observe that $H$ is continuous and that our choice of parameters can be approximated by parameters satisfying \textbf{[HX]} and \textbf{[HidA]}.

\subsection{Proof of Theorem \ref{th_lowerbound}}
\label{sec_proof_lowerbound}

First, assume that the quadratic form obtained from the second order expansion of
\begin{align*}
\Dfrak : (p,q,\hbf) \in \{ (p,q,A\fbf) \, | \, (p,q,A) \in \Cfrak \}
	\longmapsto \| g^{\pi + p, \Qbf + q, \fbf + \hbf} - g^{\pi, \Qbf, \fbf} \|_2^2
\end{align*}
is nondegenerate. Then, a careful reading of the proof of Theorem 8 of \cite{dCGL15} shows that their result can be adapted to our setting and leads to the desired minoration.

Thus, what we need to show is that \textbf{[Hdet]} (which implied the nondegeneracy of the quadratic form from $\Nfrak$) implies the nondegeneracy of the quadratic form from $\Dfrak$ (the trick being that $\Dfrak$ takes $\hbf = A \fbf$ as parameter while $\Nfrak$ takes $A$). Assume \textbf{[Hdet]}, then there exists $c_0 > 0$ such that
\begin{align*}
\text{Quad}_\Nfrak(p,q,A)
	\geq c_0 ( \| p\|_2^2 + \| \Qbf \|_F^2 + \| A \|_F^2)
\end{align*}
with the notation $\text{Quad}_\Nfrak$ referring to the quadratic form in the second order expansion of $\Nfrak$.

Since \textbf{[HidA]} holds, $\fbf$ is linearly independent, so that the application $J : A \in \Rbb^{K^* \times K^*} \longmapsto A \fbf \in \Span(\fbf)^{K^*}$ is invertible. Thus, $\hbf \longmapsto \sum_{i \in \Xcal }\| h_i \|_2^2$ and $\hbf \longmapsto \| J^{-1}(\hbf) \|_F^2$ are two norms on the same finite-dimensional linear space $\Span(\fbf)^{K^*}$, so that they are equivalent. In particular, there exists a constant $c_1 \leq 1$ such that $\| J^{-1}(\hbf) \|_F^2 \geq c_1 \sum_{i \in \Xcal }\| h_i \|_2^2$. Therefore,
\begin{align*}
\text{Quad}_\Dfrak (p,q,\hbf)
	&= \text{Quad}_\Nfrak (p,q, J^{-1}(\hbf)) \\
	&\geq c_0 c_1 \left( \| p\|_2^2 + \| \Qbf \|_F^2 + \sum_{i \in \Xcal }\| h_i \|_2^2 \right) ,
\end{align*}
which is what we wanted to prove.

\section{Proof of the control of $Z_{K,M}$}
\label{sec_proof_control_Z}

This section contains the proof of Lemma \ref{lemma_control_of_Z}.

\subsection{Concentration inequality on $Z_{K,M}(s)$}
\label{section_ineg_concentration_pour_Z}

Define for all ${\sigma > 0}$ the sets
\begin{equation*}
B_\sigma = \{ t \in S_{K,M} , \; \cinf^{L/2} \| t - s_{K,M} \|_2 \leq \sigma \}
\end{equation*}
Let $d_{g^*}$ be the semi-distance defined by $d_{g^*}^2(t_1, t_2) = \Ebb[(t_1 - t_2)^2(Z_1)] = \int g^* (t_1 - t_2)^2 \text{d}\mu^{\otimes L}$, and $d_2$ the distance induced by the norm on $\Lbf^2(\Ycal^L, \mu^{\otimes L})$.

Let $N(\epsilon, A, d) = e^{H(\epsilon, A, d)}$ denote the minimal cardinality of a covering of $A$ by brackets of size $\epsilon$ for the semi-distance $d$, that is by sets $[t_1, t_2] = \{ {t : \Ycal^L \mapsto \Rbb} \, , \, t_1(\cdot) \leq t(\cdot) \leq t_2(\cdot) \}$ such that $d(t_1, t_2) \leq \epsilon$. $H(\cdot, A, d)$ is called the \emph{bracketing entropy} of $A$ for the semi-distance $d$.

The following lemma is a Bernstein-like inequality that follows from \cite{Pau13}, Theorem 2.4:
\begin{lemma}
\label{lemme.bernstein_pour_MC}
Let $t$ be a real valued and measurable bounded function on $\Ycal^L$.  Let $V=\Ebb [t^2(Z_1)]$. There exists a positive constant $c^*$ depending only on $\Qbf^*$ and $L$ such that for all $0\leq \lambda \leq 1/(2\sqrt{2}c^* \| t \|_\infty)$ and for all $n \in \Nbb$:
\begin{equation*}
\log \Ebb \exp \left[ \lambda \sum_{s=1}^{n}\left(t(Z_{s})-\Ebb t(Z_{s})\right)  \right]
	\leq \frac{2 n c^* V\lambda^{2}}{1-2\sqrt{2}c^* \| t \|_{\infty}\lambda}
\end{equation*}
\end{lemma}

The following lemma is an extension of Theorem 6.8 from \cite{Mas07} and allows to obtain a concentration inequality on the supremum on all functions of a class when one can control its bracketing entropy.
\begin{lemma}
Let $\Xi$ be some measurable space, $(\xi_i)_{1 \leq i \leq n}$ a sequence of random variables on $\Xi$, $\Tcal$ some countable class of real valued and measurable functions on $\Xi$. Assume that there exists some positive numbers $a$ and $b$ such that for all $t \in \Tcal$, $\| t \|_{\infty} \leq b$ and $\sup_i \Ebb [t^{2}(\xi_i)]\leq a^{2}$.

Assume also that there exists some constant $C_\xi \geq 1/4$ such that for all $0\leq \lambda \leq 1/(2\sqrt{2} C_\xi b)$ and for all $t \in \Tcal$:
\begin{equation}
\label{eq:loglaplace}
\log \Ebb \exp \left[ \lambda \sum_{s=1}^{n}\left(t(\xi_{s})-\Ebb t(\xi_{s})\right)  \right] \leq \frac{2n C_\xi a^2 \lambda^{2}}{1-2\sqrt{2} C_\xi b \lambda}
\end{equation}

Assume furthermore that for any positive number $\delta$, there exists some finite set $\Bcal_\delta$ of brackets covering $\Tcal$ such that for any bracket $[t_1, t_2] \in \Bcal_\delta$, $\| t_{1}-t_{2}\|_{\infty}\leq b$ and $\sup_i \Ebb[(t_{1}-t_{2})^{2}(\xi_i)]\leq \delta^{2}$.
Let $e^{H(\delta)}$ denote the minimal cardinality of such a covering. Then, there exists a numerical constant $\kappa > 0$ such that for any measurable set $A$ such that $\Pbb(A) > 0$, 
\begin{equation*}
\Ebb^{A}\left(\sup_{t\in \Tcal}\sum_{s=1}^{n}\left(t(\xi_{s})-\Ebb t(\xi_{s}\right))\right)
	\leq \kappa C_\xi \left[
			E+a\sqrt{n\log\left(\frac{1}{\Pbb(A)}\right)}+b\log	
			\left(\frac{1}{\Pbb(A)}\right)
		\right]
\end{equation*}
where
\begin{equation*}
E=\sqrt{n}\int_0^{a}\sqrt{H(u)\wedge n}du + (b+a)H(a)
\end{equation*}
and for any measurable random variable $W$, $\Ebb^A[W] = \Ebb[W \one_A] / \Pbb(A)$.
\end{lemma}

\begin{remark}
The assumption $C_\xi \geq 1/4$ is only used to factorise the upper bound by $C_\xi$. Without it, the upper bound would be
\begin{equation*}
\kappa' \left[
	E
	+a \sqrt{C_\xi n\log\left(\frac{1}{\Pbb(A)}\right)}
	+b C_\xi \log \left(\frac{1}{\Pbb(A)}\right)
\right]
\end{equation*}

In practice, this assumption doesn't cost anything: if equation (\ref{eq:loglaplace}) holds for some constant $C_\xi$, then it holds for any constant $C' \geq C_\xi$.
\end{remark}

We will apply this lemma to $\Xi = \Ycal^L$ and $\xi_i = Z_i$. Using Lemma \ref{lemme.bernstein_pour_MC}, equation (\ref{eq:loglaplace}) holds with $C_\xi = \max(c^*, 1/4)$.

Take $\Tcal = (B_\sigma - s_{K,M}) \bigcup (- B_\sigma + s_{K,M})$, so that $\underset{t \in \Tcal}{\sup} \, \nu(t) = \underset{t \in B_\sigma}{\sup} |\nu(t - s_{K,M})|$. Then, one can check using Lemma \ref{lemme.controle_S} that the assumptions $\| t \|_{\infty} \leq b$ and $\sup_i \Ebb [t^{2}(\xi_i)]\leq a^{2}$ hold with
\begin{itemize}
\item $b = 2 \cinf^L$
\item $a = 2 \min(\sigma, \cinf^{L/2} \cquad^L)$
\end{itemize}
and $H(u) \leq \log(2) + H(u, B_\sigma - s_{K,M}, d_{g^*})$.

We can do without assuming $\Tcal$ to be countable. Indeed, $\nu$ is continuous on $\Tcal$ equipped with the infinity norm. This entails that the supremum of $\nu$ over $\Tcal$ is equal to the supremum of $\nu$ over any dense subset of ($\Tcal$, $\| \cdot \|_\infty$). Since $\Tcal \subset (\Pfrak_M)^{\otimes L}$, which is a finite dimensional metric linear space for the infinity norm, it is separable. Therefore, without loss of generality, we can get rid of the countability assumption on $\Tcal$.

Rewriting these results, we get the following lemma:
\begin{lemma}
\label{lemme_inegalite_concentration}
There exists a constant $C^*$ depending only on $\Qbf^*$ and $L$ such that for all $\sigma > 0$, for all measurable $A$ such that $\Pbb(A) > 0$:
\begin{equation*}
\Ebb^A \left(\underset{t \in B_\sigma}{\sup} |\nu(t - s_{K,M})| \right)
	\leq C^* \left[
		\frac{E}{n}
		+ \sigma \sqrt{\frac{1}{n} \log \left(\frac{1}{\Pbb(A)} \right)}
		+ \frac{2 \cinf^L}{n} \log \left(\frac{1}{\Pbb(A)} \right)
	\right]
\end{equation*}
where
\begin{align*}
E =& \sqrt{n} \int_0^\sigma \sqrt{H(u, B_\sigma - s_{K,M}, d_{g^*}) \wedge n} du + \log(2) \sigma \sqrt{n} \\
	&+ 2 (\cinf^L + \cinf^{L/2} \cquad^L) H(\sigma, B_\sigma - s_{K,M}, d_{g^*})
\end{align*}
\end{lemma}

The core of the proof consists in controlling the bracketing entropy in order to find a "good" function $\varphi$ and constants $C$ and $\sigma_{K,M}$ depending on $\cquad$, $\cinf$ and $L$ such that $x \mapsto \frac{\varphi(x)}{x}$ is nonincreasing and
\begin{equation}
\label{majoration_E}
\forall \sigma \geq \sigma_{K,M} \qquad E \leq C \varphi(\sigma) \sqrt{n}.
\end{equation}
For ease of notation, we did not write the dependency of $C$ and $\varphi$ on $K$ and $M$.

Let us see how to conclude with such an inequality.
We shall use the following result (lemma 4.23 from \cite{Mas07}).
\begin{lemma}
\label{lemme_epluchage}
Let $S$ be some countable set, $u \in S$ and $a~: S \mapsto \Rbb_+$ such that $a(u) = \inf_{t \in S} a(t)$. Let $Z$ be some process indexed by $S$ and assume that $\sup_{t \in \Bbf(\lambda)} Z(t) - Z(u)$ has finite expectation for any positive number $\lambda \geq 0$, where
\begin{equation*}
\Bbf(\lambda) = \{ t \in S, a(t) \leq \lambda \}
\end{equation*}
Then, for any function $\phi$ on $\Rbb_+$ such that $x \mapsto \phi(x)/x$ is nonincreasing on $\Rbb_+$ and satisfies for some $\lambda_* \geq 0$ to
\begin{equation*}
\forall \lambda \geq \lambda_* \geq 0, \qquad
\Ebb \left[
	\sup_{t \in \Bbf(\lambda)} Z(t) - Z(u)
\right]
	\leq \phi(\lambda)
\end{equation*}
one has for any $x \geq \lambda_*$~:
\begin{equation*}
\Ebb \left[
	\sup_{t \in S}
	\left(
		\frac{Z(t) - Z(u)}{a(t)^2 + x^2}
	\right)
\right]
	\leq 4 \frac{\phi(x)}{x^2}
\end{equation*}
\end{lemma}

In our case,
\begin{equation*}
\begin{cases}
S = S_{K,M} - s_{K,M} \\
u = s_{K,M} \\
a(t) = \cinf^{L/2} \| t - s_{K,M} \|_2 \\
Z(t) = |\nu(t-s_{K,M})| \\
\lambda_* = \sigma_{K,M} \\
\displaystyle \phi(x) = C^* \left[
		C \frac{\varphi(x)}{\sqrt{n}}
		+ x \sqrt{\frac{1}{n} \log \left(\frac{1}{\Pbb(A)} \right)}
		+ \frac{2 \cinf^L}{n} \log \left(\frac{1}{\Pbb(A)} \right)
	\right]
\end{cases}
\end{equation*}
With this choice of $S$, $a$ and $Z$, this proposition holds even if $S$ is not countable for the same reason as in Lemma \ref{lemme_inegalite_concentration}.

It follows that for all $x \geq \sigma_{K,M}$:
\begin{align*}
\Ebb^A \left[ \sup_{t \in S_{K,M}} \frac{|\nu(t-s_{K,M})|}{\cinf^L \| t - s_{K,M} \|^2_2 + x^2} \right]
	\leq 4 \frac{\phi(x)}{x^2},
\end{align*}
so that if $x_{K,M} \geq \frac{\sigma_{K,M}}{\cinf^{L/2}}$:
\begin{align*}
\Ebb^A [ Z_{K,M}(s) ]
	\leq 4 \frac{\phi(x_{K,M} \cinf^{L/2})}{x_{K,M}^2}
\end{align*}
and then
\begin{align*}
\Ebb^A [ Z_{K,M}(s) ]
	\leq& 4 \frac{C^*}{x_{K,M}^2} \Bigg[
		C \frac{\varphi(x_{K,M} \cinf^{L/2})}{\sqrt{n}}
		+ x_{K,M} \cinf^{L/2} \sqrt{\frac{1}{n} \log \left(\frac{1}{\Pbb(A)} \right)} \\
		&\qquad \qquad + \frac{2 \cinf^L}{n} \log \left(\frac{1}{\Pbb(A)} \right)
	\Bigg] \\
	=:& \psi \left( \log \left(\frac{1}{\Pbb(A)} \right) \right)
\end{align*}
Note that the function $\psi$ is nondecreasing. On the event $A = \{ Z_{K,M}(s) \geq \psi(x) \}$,
\begin{align*}
\psi(x) \leq \Ebb^A [ Z_{K,M}(s) ] \leq \psi \left( \log \left(\frac{1}{\Pbb(A)} \right) \right)
\end{align*}
so that $x \leq \log \left(\frac{1}{\Pbb(A)} \right)$ and finally $\Pbb(A) \leq e^{-x}$.

It follows that with probability $1 - e^{-z_{K,M} - z}$:
\begin{equation}
\label{majoration_Z_choix_x_a_faire}
Z_{K,M}(s) \leq 4 C^* \left[
		C \frac{\varphi(x_{K,M} \cinf^{L/2})}{x_{K,M}^2 \sqrt{n}}
		+ \cinf^{L/2} \sqrt{\frac{z_{K,M} + z}{x_{K,M}^2 n}}
		+ 2 \cinf^L \frac{z_{K,M} + z}{x_{K,M}^2 n}
	\right]
\end{equation}
and the last step of the proof will be to choose the right $x_{K,M}$ and $z_{K,M}$ (see Section \ref{sec_choixConstantes}).

\subsection{Control of the bracketing entropy}

The goal of this section is to prove equation (\ref{majoration_E}), that is to find $\varphi$, $C$ and $\sigma_{K,M}$ such that
\begin{equation*}
\forall \sigma \geq \sigma_{K,M} \qquad E \leq C \varphi(\sigma) \sqrt{n}.
\end{equation*}

The bracketing entropy is invariant under translation and increasing with respect to the inclusion relation, so
\begin{equation*}
H(u, B_\sigma - s_{K,M}, d_{g^*}) = H(u, B_\sigma, d_{g^*}) \leq H(u, S_{K,M}, d_{g^*})
\end{equation*}
Using Lemma \ref{lemme.controle_S}, we get that for all $t \in \Lbf^2(\Ycal^L, \Rbb)$, $\int t^2 g^* \text{d}\mu \leq \cinf^L \| t \|_2^2$. Therefore, a bracket of size $u / \cinf^{L/2}$ for $d_2$ is also a bracket of size $u$ for $d_{g^*}$, which implies that
\begin{equation}
H(u, B_\sigma - s_{K,M}, d_{g^*}) \leq H \left(\frac{u}{\cinf^{L/2}}, S_{K,M}, d_2 \right)
\label{eq.entropie.translation_inclusion_chgtdistance}
\end{equation}

Let us now rewrite the definition of $S_{K,M}$:
\begin{align*}
S_{K,M} &= \left\{
	\sum_{\kbf \in \{1, \dots, K\}^L}
		\pi_{k_1} \prod_{i=2}^L \Qbf_{k_{i-1}, k_i} \bigotimes_{i=1}^L f_{k_i}, \ \Qbf \in \Qcal_K, \ \pi \Qbf = \pi, \ \fbf \in (\Fcal \cap \Pfrak_M)^K
\right\}
\\
	&\subset \left\{
	\sum_{\kbf \in \{1, \dots, K\}^L}
		\mu_{\kbf} \phi_{\kbf}, \ \mu \in \Ucal, \ \phi \in \Phi
\right\}
\end{align*}
where
\begin{equation*}
\begin{cases}
\displaystyle \Ucal  = \left\{ (\pi_{k_1} \prod_{i=2}^L \Qbf_{k_{i-1}, k_i} )_{k_1,\dots,k_L}, \ \Qbf \text{ transition matrix $K \times K$}, \ \pi \geq 0, \ \pi \in \Sbb_{K-1}
\right\} \\
\displaystyle \Phi = \left\{
	(\bigotimes_{i=1}^L f_{k_i})_{k_1,\dots,k_L}, \
	\fbf \in (\Fcal \cap \Pfrak_M)^K
\right\}
\end{cases}
\end{equation*}
$\Ucal$ is equipped with the distance $d_2(a,b) = (\sum_{\kbf} (b^i_{\kbf} - a^i_{\kbf})^2)^{1/2}$. A bracket for $\Ucal$ will be a set $[a,b] = \{ c \, | \, \forall \kbf \in \{1, \dots, K\}^L, \, a_{\kbf} \leq c_{\kbf} \leq b_{\kbf} \}$.

$\Phi$ is equipped with the distance $d_{\infty,2}(u,v) = \max_\kbf \| v^i_{\kbf} - u^i_{\kbf}\|_2$. A bracket $\Phi$ will be a set $[u,v] = \{ t \, | \, \forall \kbf \in \{1, \dots, K\}^L, \, u_{\kbf}(\cdot) \leq t_{\kbf}(\cdot) \leq v_{\kbf}(\cdot) \}$.

Controlling the bracketing entropy on each of these sets will allow to control the bracketing entropy of $S_{K,M}$. Let us start with them:
\begin{lemma}
\label{lemme.entropie.entropie_matrice}
There exists a bracket covering $\{[a^i, b^i]\}_{1 \leq i \leq N_{\Ucal}(\epsilon)}$ of size $\epsilon$ of $\Ucal$ for the distance $d_2$ with cardinality
\begin{equation}
N_{\Ucal}(\epsilon) \leq \max \left( \frac{2 L K^{L/2}}{\epsilon} , 1 \right)^{K^2 - 1}
\label{eq.entropie.entropie_matrice}
\end{equation}
such that for all $i$ and $\kbf$, $0 \leq a^i_{\kbf} \leq 1$.
\end{lemma}
\begin{lemma}
\label{lemme.entropie.entropie_emission}
There exists a bracket covering $\{[u^i, v^i]\}_{1 \leq i \leq N_\Phi(\epsilon)}$ of size $\epsilon$ of $\Phi$ for the distance $d_{\infty,2}$ of cardinality
\begin{equation}
N_{\Phi}(\epsilon) \leq \max \left( \frac{L (4 \cquad M)^L}{\epsilon}, 1 \right)^{M K}
\label{eq.entropie.entropie_emission}
\end{equation}
such that
\begin{equation*}
\max_i \max_{\kbf} \| v^i_{\kbf} \|_2^2 \leq (8 M^2 \cquad^2)^L.
\end{equation*}
\end{lemma}

Let us take such bracketings and consider the following set of brackets:
\begin{equation*}
\left\{
	\left[\sum_{\kbf} \Acal^{i,j}_{\kbf}
	\, , \, \sum_{\kbf} \Bcal^{i,j}_{\kbf} \right]
\right\}_{1 \leq i \leq N_{\Ucal}(\epsilon), 1 \leq j \leq N_\Phi(\epsilon)}
\end{equation*}
where
\begin{equation*}
\forall y \in \Ycal^L, \quad
\begin{cases}
\displaystyle \Acal^{i,j}_{\kbf}(y) = \min \{
	a^i_{\kbf} u^j_{\kbf}(y), b^i_{\kbf} v^j_{\kbf(y)}
\} \\
\displaystyle \Bcal^{i,j}_{\kbf}(y) = \max \{
	a^i_{\kbf} u^j_{\kbf}(y), b^i_{\kbf} v^j_{\kbf(y)}
\}
\end{cases}.
\end{equation*}

This set covers $S_{K,M}$: for all $\mu \in \Ucal$, $\phi \in \Phi$, there exists $i \in \{1, \dots, N_{\Ucal}(\epsilon)\}$ and $j \in \{1, \dots, N_\Phi(\epsilon)\}$ such that $\mu \in [a^i,b^i]$ and $\phi \in [u^j,v^j]$, and then by construction $\sum_{\kbf} \mu_{\kbf} \phi_{\kbf} \in [\sum_{\kbf} \Acal^{i,j}_{\kbf}, \sum_{\kbf} \Bcal^{i,j}_{\kbf}]$.

Let us now bound the size of these brackets. Let $[a,b] \in \{ [a^i, b^i] \}_{1 \leq i \leq N_{\Ucal}(\epsilon)}$ and $[u,v] \in \{[u^i, v^i]\}_{1 \leq i \leq N_\Phi(\epsilon)}$, then if one denotes by $[\Acal, \Bcal]$ the corresponding bracket, there exists $(\sigma_\kbf)_{\kbf} \in \{-1, 1\}^{K^L}$ such that:
\begin{align*}
\left\|
	\sum_{\kbf} \Acal_{\kbf}
		- \sum_{\kbf} \Bcal_{\kbf}
\right\|_2^2
	&=
	\left\|
		\sum_{\kbf} \sigma_\kbf (b_{\kbf} v_{\kbf}
		- a_{\kbf} u_{\kbf} )
	\right\|_2^2 \\
	&\leq
	\left\|
		\sum_{\kbf} | b_{\kbf} v_{\kbf}
		- a_{\kbf} u_{\kbf} |
	\right\|_2^2 \\
	&\leq K^L \sum_{\kbf} \| a_{\kbf} u_{\kbf}
		- b_{\kbf} v_{\kbf} \|_2^2 \\
	&= K^L \sum_{\kbf} \| (a_{\kbf} - b_{\kbf}) v_{\kbf}
		+ a_{\kbf} (u_{\kbf} - v_{\kbf}) \|_2^2 \\
	&\leq 2 K^L \left(\sum_{\kbf} \| (a_{\kbf} - b_{\kbf}) v_{\kbf}\|_2^2
		+ \sum_{\kbf} \| a_{\kbf} (u_{\kbf} - v_{\kbf}) \|_2^2 \right) \\
	&= 2 K^L \left(\sum_{\kbf} (a_{\kbf} - b_{\kbf})^2 \| v_{\kbf}\|_2^2
		+ \sum_{\kbf} a_{\kbf}^2 \| u_{\kbf} - v_{\kbf} \|_2^2 \right).
\end{align*}

Then, by definition of the brackets, $\| u_{\kbf} - v_{\kbf} \|_2^2 \leq \epsilon^2$ and $\sum_{\kbf}(a_{\kbf} - b_{\kbf})^2 \leq \epsilon^2$. In addition, we assumed $| a_{\kbf} | \leq 1$ and $\| v_\kbf \|_2^2 \leq (8 M^2 \cquad^2)^L$ for all $\kbf$, so that
\begin{align*}
\|
	\sum_{\kbf} \Acal_{\kbf}
		- \sum_{\kbf} \Bcal_{\kbf}
\|_2^2
	&\leq 2 K^L \epsilon^2 ( (8 M^2 \cquad^2)^L + K^L ),
\end{align*}
which implies
\begin{multline}
N(\epsilon, S_{K,M}, d_2) \leq
	N_\Ucal \left(\frac{\epsilon}{\sqrt{2 K^L (K^L + (8 M^2 \cquad^2)^L)}} \right) \\
	\times N_\Phi \left(\frac{\epsilon}{\sqrt{2 K^L (K^L + (8 M^2 \cquad^2)^L)}} \right),
\label{eq.entropie.decoupage_matrice_emission}
\end{multline}
and finally by combining \ref{eq.entropie.translation_inclusion_chgtdistance}, \ref{eq.entropie.entropie_matrice}, \ref{eq.entropie.entropie_emission} and \ref{eq.entropie.decoupage_matrice_emission}:
\begin{align*}
H(u, B_\sigma - s_{K,M}, d_{g^*}) &\leq 
(K^2 - 1) \log \max \left( \frac{\cinf^{L/2} \sqrt{2 K^L (K^L + (8 M^2 \cquad^2)^L)} 2 L K^{L/2} }{u}, 1 \right) \\
&+ M K \log \max \left( \frac{\cinf^{L/2} \sqrt{2 K^L (K^L + (8 M^2 \cquad^2)^L)} L (4 \cquad M)^L}{u}, 1 \right) \\
\leq (MK + K^2 - 1) &\log \max \left( \frac{\cinf^{L/2} \sqrt{2 (K^L + (8 M^2 \cquad^2)^L)} L K^L (4 \cquad M)^L}{u},1 \right) \\
\leq (MK + K^2 - 1) &\log \max \left( \frac{\cinf^{L/2} \sqrt{2 (n^L + (8 \cquad^2)^L n^{2L})} n n^L (4 \cquad)^L n^L}{u},1 \right) \\
\leq (MK + K^2 - 1) &\log \max \left( \frac{\cinf^{L/2} 2 (8 \cquad^2)^{L/2} n^{L} n n^L (4 \cquad)^L n^L}{u},1 \right) \\
\leq (MK + K^2 - 1) &\log \max \left( \frac{2 (16 \cinf^{1/2} \cquad^2)^L n^{3L+1}}{u},1 \right) \\
\leq (MK + K^2 - 1) &\log \max \left( \frac{n^{6L}}{u},1 \right)
\end{align*}
for $n$ large enough ($n \geq n_0 := 16 \cinf^{1/2} \cquad^2$) because we assumed $M \leq n$, $K \leq n$, $L \leq n$, $\cquad \geq 1$ and $\cinf \geq 1$. Thus, Lemma \ref{lemme.entropie_et_phi} implies that if we write $C_0 = \sqrt{\pi}$ and
\begin{align*}
\varphi(\sigma) = C_0 \sigma \sqrt{MK + K^2 - 1} \left(1 + \sqrt{\log \left( \max \left\{ \frac{n^{6L}}{\sigma} ,1 \right\} \right) } \right)
\end{align*}
then for all $n \geq n_0$ and $\sigma > 0$,
\begin{align*}
\begin{cases}
\displaystyle \sigma^2 H(\sigma, S_{K,M}, d_2)
	\leq \varphi(\sigma)^2 \\
\displaystyle \int_0^\sigma \sqrt{H(u, S_{K,M}, d_2)} du
	\leq \varphi(\sigma) \\
\log(2) \sigma \leq \varphi(\sigma)
\end{cases}
\end{align*}
Let us now check that this function $\varphi$ satisfies equation (\ref{majoration_E}). First, note that $x \mapsto \frac{\varphi(x)}{x}$ is nonincreasing, so that $x \mapsto \frac{\varphi(x)}{x^2}$ is also nonincreasing. Thus, we may define $\sigma_{K,M}$ as the unique solution of the equation $\varphi(x) = \sqrt{n} x^2$, and then for all $\sigma \geq \sigma_{K,M}$:
\begin{equation*}
H(\sigma, B_\sigma - s_{K,M}, d_{g^*}) \leq \frac{\varphi(\sigma)^2}{\sigma^2} \leq \frac{\varphi(\sigma)}{\sigma} \sigma \sqrt{n} = \varphi(\sigma) \sqrt{n}
\end{equation*}
Equation (\ref{majoration_E}) follows immediately with $C = 2 (1 + \cinf^L + \cinf^{L/2} \cquad^L)$.

\paragraph{Proof of Lemma \ref{lemme.entropie.entropie_matrice}}

Let $\epsilon \in (0,2)$.

We start with the family $\{ [k/n, (k+1)/n], \; k \in \{0, \dots, n-1 \} \}$ with $n$ an integer between $1 / \epsilon$ and $2 / \epsilon$, which gives a bracket covering of size $\epsilon$ of $[0,1]$ with cardinality smaller than $2 / \epsilon$. These brackets will be used to control each free component of $\Qbf$ and $\pi$, that is $K^2 - 1$ components.

More precisely, we define the following bracket set:
\begin{align*}
\Big\{& [A,B] \; | \;
A_{\kbf} = \frac{1}{n^L} p_{k_1} \prod_{i=2}^L a_{k_{i-1}, k_i}, \;
B_{\kbf} = \frac{1}{n^L} (p_{k_1} + 1) \prod_{i=2}^L (a_{k_{i-1}, k_i} + 1),
\\
&\qquad p \in \{ 0, \dots, n-1 \}^{K - 1}, \;
\sum_{k=1}^{K-1} p_k < n, \;
p_K = n - \sum_{k=1}^{K-1} (p_k + 1),
\\
&\qquad a \in \{ 0, \dots, n-1 \}^{K \times (K-1)}, \;
\forall i \in \{ 1, \dots, K \}, \; \sum_{k=1}^{K-1} a_{i,k} < n \text{ and } \\
&\qquad a_{i,K} = n - \sum_{k=1}^{K-1} (a_{i,k} + 1)
\Big\}.
\end{align*}
This set covers $\Ucal$ and its cardinality is smaller than $\displaystyle \left(\frac{2}{\epsilon} \right)^{K^2 - 1}$. To get the bracket's size, note that
\begin{align*}
\sum_{\kbf \in \{1, \dots, K\}^L} & \left(
	\frac{1}{n^L} p_{k_1} \prod_{i=2}^L a_{k_{i-1}, k_i}
	- \frac{1}{n^L} (p_{k_1} + 1) \prod_{i=2}^L (a_{k_{i-1}, k_i} + 1)
	 \right)^2 \\
	&= \frac{1}{n^{2L}} \sum_{\kbf \in \{1, \dots, K\}^L} \left(
	\prod_{i=2}^L a_{k_{i-1}, k_i}
	+ \sum_{j=2}^L p_{k_1} \prod_{i \neq j, i \geq 2} a_{k_{i-1}, k_i}
	\right)^2 \\
		&\leq \frac{L^2 n^{2L-2} K^L}{n^{2L}} \\
		&\leq L^2 K^L \epsilon^2,
\end{align*}
and in the end
\begin{align*}
N(u, \Ucal, d_2) \leq \max \left( \frac{L K^{L/2}}{u} , 1 \right)^{K^2 - 1}.
\end{align*}

\paragraph{Proof of Lemma \ref{lemme.entropie.entropie_emission}}

All $f \in \Fcal \cap \Pfrak_M$ can be written as $\sum_{m=1}^M \lambda_m \varphi_m$ where $(\varphi_m)_{m \in \{1, \dots, M\}}$ is an orthonormal basis of $\Pfrak_M$. Then, assumption \textbf{[HF]} implies that $| \lambda_m | \leq \cquad$ for all $m \in \{1, \dots, M\}$.

We will therefore start from a bracket covering of the euclidian ball of radius $\cquad$ of $\Rbb^M$, from which we will construct a covering of $\Fcal \cap \Pfrak_M$ and of $\Phi$.
\begin{lemma}
\label{lemme.recouvrement_boule_euclidienne}
Let $\epsilon \in (0,4)$. There exists a bracket covering $\{ [a^i, b^i] \}_{1 \leq i \leq N_M}$ of size $\epsilon$ of the euclidian ball of radius $\cquad$ of $\Rbb^M$ with cardinality
\begin{equation*}
N_M \leq \max \left( \frac{4 \cquad \sqrt{M}}{\epsilon}, 1 \right)^M
\end{equation*}
such that for all $m \in \{1, \dots, M\}$, $i \in \{1, \dots, N_{M} \}$, $- \cquad \leq a^i_m \leq b^i_m \leq \cquad$.
\end{lemma}
\begin{proof}
We start with a bracket covering of size $\epsilon/\sqrt{M}$ of the infinity ball of radius $\cquad$ of $\Rbb^M$. This can be done by a regular partition with $\max(\lceil 2\cquad / \epsilon \rceil, 1)$ pieces along each coordinate. One can easily check that such a covering is also a covering of size $\epsilon$ of the euclidian ball of radius $\cquad$ of $\Rbb^M$. To conclude, it is enough to notice that $\lceil x \rceil \leq 2x$ as soon as $x > 1/2$, and that $2\cquad / \epsilon > 1/2$ because $\cquad \geq 1$ and $\epsilon < 4$.
\end{proof}

Let $\{ [a^i, b^i] \}_{1 \leq i \leq N_M}$ be such a covering. For all $m \in \{1, \dots, M\}$, $i \in \{1, \dots, N_{M} \}$ and $y \in \Ycal$, let
\begin{align*}
u^i_m(y) &=
\begin{cases}
a^i_m \qquad \text{if } \varphi_m(y) \leq 0 \\
b^i_m \qquad \text{otherwise}
\end{cases} \\
v^i_m(y) &= a^i_m + b^i_m - u^i_m(y)
\end{align*}
and for all $i \in \{1, \dots, N_{M} \}$ and $y \in \Ycal$,
\begin{equation*}
\begin{cases}
\displaystyle U^i_1(y) = \sum_{m=1}^M u^i_m(y) \varphi_m(y) \\
\displaystyle U^i_2(y) = \sum_{m=1}^M v^i_m(y) \varphi_m(y)
\end{cases}
\end{equation*}
and finally for all $ \ibf = (i_1, \dots, i_K) \in \{1, \dots, N_{M} \}^K$ and $\kbf = (k_1, \dots, k_L) \in \{1, \dots, K\}^L$:
\begin{equation*}
\begin{cases}
\displaystyle (V^{\ibf})_\kbf = \min \left\{
	\bigotimes_{\beta=1}^L U^{i_{k_\beta}}_{\sigma_\beta};
	\quad \sigma \in \{1,2\}^L
\right\} \\
\displaystyle (W^{\ibf})_\kbf = \max \left\{
	\bigotimes_{\beta=1}^L U^{i_{k_\beta}}_{\sigma_\beta};
	\quad \sigma \in \{1,2\}^L
\right\}
\end{cases} .
\end{equation*}
It is enough to show that $\{ [V^{\ibf}, W^{\ibf}], \; \ibf \in \{1, \dots, N_{M} \}^K \}$ is a bracket covering of size $L (4\cquad M)^{L-1} \sqrt{M} \epsilon$ of $\Phi$ that satisfies
\begin{equation*}
\max_i \max_{\kbf} \int (W^i_{\kbf})^2 \text{d}\mu^{\otimes L} \leq (8 M^2 \cquad^2)^L.
\end{equation*}
Applying the Cauchy-Schwarz inequality, one gets that for all $i \in \{1, \dots, N_{M} \}$,
\begin{align*}
\|U^i_2 - U^i_1\|_2^2
	&=\| \sum_{m=1}^M |b^i_{m} - a^i_{m}|. |\varphi_{m}| \|_2^2 \\
	&\leq M \|b^i-a^i\|_2^2 \\
	&\leq M \epsilon^2.
\end{align*}
Moreover, for all $i \in \{1, \dots, N_{M} \}$ and $\sigma \in \{1,2\}$,
\begin{align*}
\|U^{i}_\sigma\|_{2}^{2}
	&\leq\| \sum_{m_i=1}^M |b^i_{m} + a^i_{m}|. |\varphi_{m}| \|_{2}^{2}\\
	&\leq 2 M (\|a^i\|_{2}^{2}+\|b^i\|_{2}^{2}) \\
	&\leq 4 M^2 \cquad^2.
\end{align*}
We then use that for all $\ibf \in \{1, \dots, N_{M} \}^K$ and $\kbf \in \{1, \dots, K\}^L$,
\begin{align*}
| W^\ibf_\kbf - V^\ibf_\kbf |(y)
	\leq& \sum_{\gamma = 1}^L
		\left| U_2^{i_{k_\gamma}} - U_1^{i_{k_\gamma}} \right| (y_\gamma)
		\max_{\jbf \in \{1,2\}^{L}} 
			\prod_{\beta \neq \gamma, \beta = 1}^L
				\left| U_{j_\beta}^{i_{k_\beta}} \right| (y_\beta) \\
	\leq& \sum_{\gamma = 1}^L
		\left| U_2^{i_{k_\gamma}} - U_1^{i_{k_\gamma}} \right| (y_\gamma)
		\prod_{\beta \neq \gamma, \beta = 1}^L \left(
			\left| U_1^{i_{k_\beta}}\right|
			+ \left| U_2^{i_{k_\beta}} \right|
		\right) (y_\beta)
\end{align*}
so that
\begin{align*}
\| W^\ibf_\kbf - V^\ibf_\kbf \|_2^2
	\leq& L \sum_{\gamma = 1}^L
		\left\| U_2^{i_{k_\gamma}} - U_1^{i_{k_\gamma}} \right\|_2^2
		\prod_{\beta \neq \gamma, \beta = 1}^L 2 \left(
			\left\| U_1^{i_{k_\beta}}\right\|_2^2
			+ \left\| U_2^{i_{k_\beta}} \right\|_2^2
		\right) \\
	\leq& L 2^{L-1} \sum_{\gamma = 1}^L
		M \epsilon^2
		\prod_{\beta \neq \gamma, \beta = 1}^L
			(2 \times 4 M^2 \cquad^2) \\
	=& L^2 (16 M^2 \cquad^2)^{L-1} M \epsilon^2 \\
	=& (L (4\cquad M)^{L-1} \sqrt{M} \epsilon)^2
\end{align*}
and finally $d_{\infty,2}(W^\ibf, V^\ibf) \leq L (4\cquad M)^{L-1} \sqrt{M} \epsilon$ for all $\ibf \in \{1, \dots, N_{M} \}^K$.

The last part of the lemma is proved by noting that for all $\ibf$ and $\kbf$,
\begin{align*}
(W^\ibf)_{\kbf}^2
	&= \max \{
			\bigotimes_{\beta=1}^L (U^{i_{k_\beta}}_{\sigma_\beta})^2;
			\quad \sigma \in \{1,2\}^L
		\} \\
	&\leq \sum_{\sigma \in \{1,2\}^L}
			\bigotimes_{\beta=1}^L (U^{i_{k_\beta}}_{\sigma_\beta})^2,
\end{align*}
so that
\begin{align*}
\int (W^\ibf)_{\kbf}^2 \text{d}\mu^{\otimes L}
	&\leq \sum_{\sigma \in \{1,2\}^L}
			\prod_{\beta=1}^L \| U^{i_{k_\beta}}_{\sigma_\beta} \|_2^2 \\
	&\leq \sum_{\sigma \in \{1,2\}^L} (4 M^2 \cquad^2)^L \\
	&\leq (8 M^2 \cquad^2)^L.
\end{align*}

\subsection{Choice of parameters}
\label{sec_choixConstantes}

Let us come back to equation (\ref{majoration_Z_choix_x_a_faire}). Since $x \mapsto \frac{\varphi(x)}{x}$ is nonincreasing, one has $\frac{\varphi(x_{K,M} \cinf^{L/2})}{x_{K,M} \sqrt{n}} \leq \sigma_{K,M} \cinf^{L/2}$ as soon as $x_{K,M} \geq \frac{\sigma_{K,M}}{\cinf^{L/2}}$, so with probability $1 - e^{-z_{K,M} - z}$:
\begin{equation*}
Z_{K,M}(s) \leq 4 C^* \left[
		C \cinf^{L/2} \frac{\sigma_{K,M}}{x_{K,M}}
		+ \cinf^{L/2} \sqrt{\frac{z_{K,M} + z}{x_{K,M}^2 n}}
		+ 2 \cinf^L \frac{z_{K,M} + z}{x_{K,M}^2 n}
	\right].
\end{equation*}
Let $C' = C^* \max(C,1) \cinf^L$. One gets
\begin{equation*}
Z_{K,M}(s) \leq 4 C' \left[
		\frac{\sigma_{K,M}}{x_{K,M}}
		+ \sqrt{\frac{z_{K,M} + z}{x_{K,M}^2 n}}
		+ \frac{z_{K,M} + z}{x_{K,M}^2 n}
	\right].
\end{equation*}

Let $x_{K,M} = \theta^{-1} \sqrt{\sigma_{K,M}^2 + \frac{z_{K,M} + z}{n}}$ with $\theta$ such that $2 \theta + \theta^2 \leq 1 / (16 C')$. Then, with probability $1 - e^{-z_{K,M} - z}$:
\begin{equation*}
Z_{K,M}(s) \leq 4 C'(\theta + \theta + \theta^2) \leq \frac{1}{4}
\end{equation*}
Now choose $z_{K,M} = M + K$, it follows that $\sum_{K \in \Nbb^*, M \in \Mcal} e^{-z_{K,M}} \leq (e-1)^{-2} \leq 1$ and the first point of the lemma is proved.

Moreover, one has with probability $1 - e^{-z}$, for all $K,M$:
\begin{align*}
Z_{K,M}(s) x_{K,M}^2
	&\leq 4 C' \left[
		\sigma_{K,M} x_{K,M}
		+ x_{K,M} \sqrt{\frac{z_{K,M} + z}{n}}
		+ \frac{z_{K,M} + z}{n}
	\right] \\
	&\leq 4 C' \left[
		2 \theta x_{K,M}^2
		+ \frac{z_{K,M} + z}{n}
	\right] \\
	&= 4 C' \left[
		2 \theta^{-1} \sigma_{K,M}^2
		+ (2 \theta^{-1} + 1) \frac{M+K}{n}
		+ (2 \theta^{-1} + 1) \frac{z}{n}
	\right]
\end{align*}
Let $A = 4C' (2 \theta^{-1} + 1)$. We get that with probability $1 - e^{-z}$, for all $K,M$:
\begin{align*}
Z_{K,M}(s) x_{K,M}^2
	&\leq A \left[
		\sigma_{K,M}^2
		+ \frac{M+K}{n}
		+ \frac{z}{n}
	\right]
\end{align*}
Therefore the lemma holds as soon as
\begin{equation}
\label{eq_condition_pen1}
\forall K \leq n, \; \forall M \leq n, \quad
\pentilde(n,M,K) \geq A \left[
		\sigma_{K,M}^2
		+ \frac{M+K}{n} \right]
\end{equation}

\begin{lemma}
There exists constants $C_1$ and $n_1$ such that for all $n \geq n_1$:
\begin{equation*}
\sigma_{K,M} \leq C_1 \sqrt{\frac{M K + K^2 - 1}{n}} (1 + \sqrt{\log(n)})
\end{equation*}
\end{lemma}

\begin{proof}
Let $x(C) = C \sqrt{\frac{M K + K^2 - 1}{n}} (1 + \sqrt{\log(n)})$.

$\sigma_{K,M}$ is defined by the equation $\frac{\varphi(x)}{x^2 \sqrt{n}} = 1$. The function $x \mapsto \frac{\varphi(x)}{x^2}$ is nondecreasing, so it is enough to show that $\frac{\varphi(x(C))}{x(C)^2 \sqrt{n}} \leq 1$ for some constant $C$ that we can assume to be greater than 1.

It is easy to check that there exists a constant $n_1$ such that for all $n \geq n_1$, $\frac{\varphi(n^{6L})}{(n^{6L})^2 \sqrt{n}} \leq 1$, so that $\sigma_{K,M} \leq n^{6L}$, which makes it possible to assume $x(C) \leq n^{6L}$. Then
\begin{align*}
\frac{\varphi(x(C))}{x(C)^2 \sqrt{n}}
	&= \frac{C_0}{C} \frac{1 + \sqrt{\log \left( \frac{n^{6L+1/2}}{ C \sqrt{(M K + K^2 - 1)} \left(1 + \sqrt{\log(n)} \right)} \right)}}{1 + \sqrt{\log(n)}} \\
	&\leq \frac{C_0}{C} \frac{1 + \sqrt{\log(n^{7L})}}{1 + \sqrt{\log(n)}} \\
	&= \frac{C_0}{C} \frac{1 + \sqrt{7L}\sqrt{\log(n)}}{1 + \sqrt{\log(n)}}
\end{align*}
and by taking $C_1 = \max(C_0 \sqrt{7 L}, 1)$, one gets that
\begin{align*}
\frac{\varphi(x(C_1))}{x(C_1)^2 \sqrt{n}} \leq 1
\end{align*}
which means that $\sigma_{K,M} \leq x(C_1)$.
\end{proof}

The condition of equation (\ref{eq_condition_pen1}) becomes
\begin{equation*}
\pentilde(n,M,K) \geq A \left[ \frac{C_1^2 (MK + K^2 - 1)(1 + \sqrt{\log(n)})^2 + M+K}{n} \right]
\end{equation*}
which is implied by
\begin{equation*}
\pentilde(n,M,K) \geq \rho (MK + K^2 - 1) \frac{\log(n)}{n}
\end{equation*}
for some constant $\rho$ depending only on $\cquad$, $\cinf$, $\Qbf^*$ and $L$. This concludes the proof.

\section{}
\label{app_auxiliary}

\subsection{Proof of Proposition \ref{prop_overpenalizing}}
\label{sec_proof_prop_overpenalizing}

Let $m \geq 3$. Note $r = \frac{m}{m-1}$ and $K_0 = (m-1)^m$. One can check that $K = K_0 r^m \geq 2 K_0$ and $K_0 r^k \in \Nbb^*$ for all $k \in \{0, \dots, m\}$.

Denote by $n(K)$ the integer $n_1$ in the hypothesis \textbf{[Hpen]}$(0,\rho)$ corresponding to $K^* = K$.
Then for all $n \geq \sup_{k \in \{0, \dots, m-1\}} n(K_0 r^k)$, for all $M$ and for all $k \in \{1, \dots, m\}$,
\begin{multline*}
\pen(n, M, K_0 r^k) - \pen(n, M, K_0 r^{k-1})
	\geq \rho (M K_0 r^k + K_0^2 (r^2)^k  - 1) \frac{\log(n)}{n}.
\end{multline*}
Taking the sum over $k \in \{1, \dots, m\}$, one gets that
\begin{align*}
\pen(n, M, K)
	&\geq \rho \left(M \frac{r}{r-1}(K-K_0) + \frac{r^2}{r^2-1} (K^2 - K_0^2) - m \right) \frac{\log(n)}{n} \\
	&\geq \rho \left(M \frac{r}{r-1}(K-K_0) + \frac{r^2}{r^2-1} (K^2 - 2 K_0^2) \right) \frac{\log(n)}{n}
\end{align*}
since $m \leq K_0^2 = (m-1)^{2m}$. Using that $K \geq 2K_0$,
\begin{align*}
\pen(n, M, K)
	\geq \frac{\rho}{2}
		\left(\frac{r}{r-1} M K + \frac{r^2}{r^2-1} K^2 \right) \frac{\log(n)}{n}.
\end{align*}
Let $v_m = \frac{\rho}{2} \min\left(\frac{r}{r-1}, \frac{r^2}{r^2-1}\right)$. One gets
\begin{align*}
\pen(n, M, K)
	&\geq v_m ( M K + K^2 ) \frac{\log(n)}{n}.
\end{align*}
Therefore, there exists a non-decreasing sequence $(u_n)_{n \geq 1}$ such that
\begin{equation*}
\begin{cases}
\forall n, \; \forall M \leq n, \; \forall K \leq n, \; \pen(n, M, K) \geq u_n ( M K + K^2 - 1 ) \frac{\log(n)}{n} \\
\forall m, \; u_{\max (m, \sup_{k \in \{0, \dots, m-1\}} n(K_0 r^k))} \geq v_m
\end{cases}.
\end{equation*}
and since $v_m \longrightarrow \infty$, we get that $u_n \longrightarrow \infty$, which concludes the proof.

We could for instance take
\begin{equation*}
u_n = \max \left( 0, \, \sup \left\{ v_i \; \Big| \; i \leq n \text{ s.t. } \sup_{k \in \{0, \dots, i-1\}} n(i^k (i-1)^{i-k}) \leq n \right\} \right).
\end{equation*}

\subsection{Auxiliary lemmas}

\begin{lemma}
\label{lemme.controle_S}
\begin{equation*}
\forall t \in \bigcup_K S_K, \quad
\begin{cases}
\| t \|_\infty \leq \cinf^L \\
\| t \|_2 \leq \cquad^L \\
\Ebb [ t^2 ] \leq \cinf^L \| t \|^2_2
\end{cases}
\end{equation*}
\end{lemma}

\begin{proof}
$t$ can be written as $t = g^{\pi, \Qbf, \fbf}$ with $\pi$ a probability $K$-uple, $\Qbf$ a transition matrix of size $K$ and $\fbf \in \Fcal^K$ for some $K \geq 1$.

The first point follows from
\begin{align*}
\|t\|_\infty =& \left\| \sum_{k_1, \dots, k_L = 1}^K \pi(k_1) \prod_{i=2}^{L} \Qbf(k_{i-1}, k_i) \bigotimes_{i=1}^L f_{k_i} \right\|_\infty \\
	\leq& \sum_{k_1, \dots, k_L = 1}^K \pi(k_1) \prod_{i=2}^{L} \Qbf(k_{i-1}, k_i) \left\| \bigotimes_{i=1}^L f_{k_i} \right\|_\infty \\
	\leq& \sum_{k_1, \dots, k_L = 1}^K \pi(k_1) \prod_{i=2}^{L} \Qbf(k_{i-1}, k_i) \prod_{i=1}^L \| f_{k_i} \|_\infty \\
	\leq& \cinf^L \sum_{k_1, \dots, k_L = 1}^K \pi(k_1) \prod_{i=2}^{L} \Qbf(k_{i-1}, k_i) \\
	=& \cinf^L
\end{align*}

\allowdisplaybreaks
For the second point, we use the Cauchy-Schwarz inequality:
\begin{align*}
\|t\|^2_2 =& \int \left( \sum_{k_1, \dots, k_L = 1}^K \pi(k_1) \prod_{i=2}^{L} \Qbf(k_{i-1}, k_i) \prod_{i=1}^L f_{k_i}(y_i) \right)^2 \text{d}\mu(y_1) \dots \text{d}\mu(y_L) \\
	=& \int \Bigg( \sum_{k_1, \dots, k_L = 1}^K \sqrt{\pi(k_1) \prod_{i=2}^{L} \Qbf(k_{i-1}, k_i)}  \\
		&\qquad \left( \sqrt{\pi(k_1) \prod_{i=2}^{L} \Qbf(k_{i-1}, k_i)} \prod_{i=1}^L f_{k_i}(y_i) \right) \Bigg)^2 \text{d}\mu(y_1) \dots \text{d}\mu(y_L) \\
	\leq& \int \left( \sum_{k_1', \dots, k_L' = 1}^K \pi(k_1') \prod_{i=2}^{L} \Qbf(k_{i-1}', k_i') \right) \\
		&\qquad \left( \sum_{k_1, \dots, k_L = 1}^K \pi(k_1) \prod_{i=2}^{L} \Qbf(k_{i-1}, k_i) \prod_{i=1}^L f^2_{k_i}(y_i) \right) \text{d}\mu(y_1) \dots \text{d}\mu(y_L) \\
	=& \sum_{k_1, \dots, k_L = 1}^K \pi(k_1) \prod_{i=2}^{L} \Qbf(k_{i-1}, k_i) \int \prod_{i=1}^L f^2_{k_i}(y_i) \text{d}\mu(y_1) \dots \text{d}\mu(y_L) \\
	=& \sum_{k_1, \dots, k_L = 1}^K \pi(k_1) \prod_{i=2}^{L} \Qbf(k_{i-1}, k_i) \prod_{i=1}^L \| f_{k_i}\|_2^2 \\
	\leq& \sum_{k_1, \dots, k_L = 1}^K \pi(k_1) \prod_{i=2}^{L} \Qbf(k_{i-1}, k_i) \cquad^{2L} \\
	=& \cquad^{2L} \\
\end{align*}

The last point comes from
\begin{align*}
\Ebb [ t^2 ] =& \int g^* t^2 \text{d}\mu^{\otimes L} \\
	\leq& \int \| g^* \|_\infty t^2 \text{d}\mu^{\otimes L} \\
	\leq& \cinf^L \int t^2 \text{d}\mu^{\otimes L} \quad \text{par le premier point} \\
	=& \cinf^L \| t \|^2_2
\end{align*}
\end{proof}

\begin{lemma}
\label{lemme.entropie_et_phi}
Let $A, B \in \Rbb_+^*$. Let $H : x \in \Rbb_+^* \mapsto A \log \max(\frac{B}{x}, 1)$, and $\varphi(x) : x \in \Rbb_+^* \mapsto x \sqrt{\pi A} (1 + \sqrt{\log \max(\frac{B}{x}, 1)})$. Then:
\begin{equation*}
\begin{cases}
\displaystyle x^2 H(x) \leq \varphi(x)^2 \\
\displaystyle \int_0^x \sqrt{H(u)} du \leq \varphi(x)
\end{cases}
\end{equation*}
\end{lemma}

\begin{proof}
The first point is straightforward.

For the second point, we have two cases.
\begin{itemize}
\item[\textbf{Case 1:}] $x \leq B$. Then $H(x) = \log (\frac{B}{x})$. Therefore, we can use that $\int_0^{\sigma}\sqrt{\log(\frac{B}{x})}dx\leq \sigma(\sqrt{\pi}+\sqrt{\log(\frac{B}{\sigma})})$, which is enough to conclude.

\item[\textbf{Case 2:}] $x \geq B$. Then $H(x) = 0$ and $\varphi(x) = x \sqrt{\pi A} \geq B \sqrt{\pi A} = \varphi(B)$. Thus,
\begin{align*}
\int_0^x \sqrt{H(u)} du &= \int_0^B \sqrt{H(u)} du \\
	&\leq \varphi(B) \\\
	&\leq \varphi(x)
\end{align*}
\end{itemize}
\end{proof}

\end{document}